\documentclass[12pt]{amsart}

\usepackage{mathtools}
\mathtoolsset{showonlyrefs}
\usepackage{twoopt}
\usepackage{amssymb}
\usepackage{amsthm}
\usepackage{amsmath}
\usepackage{aurical}
\usepackage{calligra}
\usepackage{graphicx}
\usepackage{comment}
\usepackage[usenames,dvipsnames]{color}
\usepackage[bookmarks]{hyperref}

\title[Hookean incompressible viscoelasticity]{Global existence of\\ small displacement 
solutions for\\  Hookean incompressible viscoelasticity\\
 in 3D}

\author{Boyan Jonov}
\address{ Raytheon Technologies\\ Goleta, CA 93117}

\author
{Paul Kessenich}

\address{ Department of Mathematics\\ University of Michigan\\ Ann Arbor, MI  48109}

\author{
Thomas C.\ Sideris}
\address{Department of Mathematics\\University of California\\Santa Barbara, CA 93106}

\newtheorem{theorem}{Theorem}[section]
\newtheorem{corollary}[theorem]{Corollary}
\newtheorem{lemma}[theorem]{Lemma}
\newtheorem{prop}[theorem]{Proposition}

\theoremstyle{definition}

\theoremstyle{remark}

\theoremstyle{remark}

\numberwithin{equation}{section}

\newcommand{\eps}{\varepsilon}

\newcommand{\jb}[1]{\langle #1 \rangle}

\newcommand{\ltn}[2][L^2]{\|#2\|_{#1}}
\newcommand{\ltns}[2][L^2]{\|#2\|^2_{#1}}
\newcommand{\ltip}[2]{\langle#1,#2\rangle_{L^2}}
\newcommand{\ips}[3]{\langle#1,#2\rangle_{#3}}
\newcommand{\ip}[2]{\langle#1,#2\rangle}

\newcommand{\isum}[2][q]{\sum_{\genfrac{}{}{0pt}{}{|a|+k\le #2}{k\le#1}}}
\newcommand{\stacktwo}[2]{\genfrac{}{}{0pt}{}{ #1}{#2}}
\newcommand{\stackthree}[3]{\stacktwo{\stacktwo{#1}{#2}}{#3}}

\newcommand{\rr}{{\mathbb R}}

\newcommand{\omt}{\widetilde{\Omega}}

\newcommand{\pbar}{\bar{p}}
\newcommand{\qbar}{\bar{q}}

\newcommandtwoopt{\EE}[3][U][t]{{\mathcal E}_{#3}[#1](#2)}
\newcommand{\EEZ}[2][U_0]{\mathcal E_{#2}[#1]}

\newcommandtwoopt{\EEP}[4][U][t]{{\mathcal E}^{#3}_{#4}[#1](#2)}

\newcommandtwoopt{\EEH}[3][U][t]{{\mathcal E}_{#3}^{1/2}[#1](#2)}
\newcommand{\EEZH}[2][U_0]{\mathcal E^{1/2}_{#2}[#1]}

\newcommandtwoopt{\YYI}[3][U][t]{{\mathcal Y}_{#3}[#1](#2)}
\newcommandtwoopt{\ZZI}[3][U][t]{{\mathcal Z}_{#3}[#1](#2)}

\newcommandtwoopt{\YYIH}[3][U][t]{   {\mathcal Y}_{#3}^{1/2}    [#1](#2)  }

\newcommand{\slt}{{\rm SL}(3)}

\newcommand{\td}[1]{\widetilde{#1}}

\newcommand{\hs}[1]{\hspace{#1cm}}
\newcommand{\isp}[1]{\quad\text{#1}\quad}

\allowdisplaybreaks

\begin{document}

\begin{abstract}
The initial value problem for Hookean incompressible viscoelastic motion
in three space dimensions has global strong solutions with small displacements.
\end{abstract}
\maketitle


\section{Introduction}

This article establishes the global existence of strong solutions to the equations of motion
for viscoelastic Hookean incompressible materials in $\rr^3$, with 
small initial displacements.
The equations of motion take the form
%
%
\begin{align}
\label{PDE1}
&\partial_t F + v \cdot \nabla F - \nabla v F = 0\\
\label{PDE2}
& \partial_t v + v \cdot \nabla v + \nabla \pi - \mu\nabla \cdot(FF^T) =  \nu\Delta v\\
\label{PDE3}
&\nabla \cdot v = 0,
\end{align}
%
%
in which $(F,v,\pi):[0,T)\to \slt\times \rr^3\times\rr$ represent the deformation gradient, velocity, and pressure,
expressed in spatial coordinates.
The central point of this article is that
the   restriction \eqref{IDHypoth} placed on the size of the initial data 
$\left.(F-I,v)\right|_{t=0}$ is to be uniform with respect to the 
Reynolds number $\text{Re}=\nu^{-1}>1$.
This system can be viewed as the
Oldroyd-B model at infinite Weissenberg number with conformation tensor $\sigma=FF^\top$
arising in the theory of complex fluids.

Global existence  of small displacement solutions for incompressible isotropic elastic materials ($\nu=0$) was obtained in \cite{Sideris-Thomases-2005}
and \cite{Sideris-Thomases-2007}.  Thus, the  result in this paper establishes the stability  of a
 nonlinear elastodynamic system under a viscous perturbation. 
 An earlier version of the results in this paper appeared in \cite{Kessenich-2009}.
  The restriction to Hookean materials is made here 
only for technical simplification.  
A scalar analog  involving 
a dissipative perturbation of a scalar nonlinear wave equation with 
nearly null structure was presented in \cite{Jonov-Sideris-2015}.

On the other hand, 
the system reduces to the Navier-Stokes equations,
under  the hydrodynamical limit obtained by sending the elastic modulus $\mu\to0$.
In $\rr^3$, the Navier-Stokes equations have a global strong solution for initial velocities which are small relative to $\nu=\text{Re}^{-1}$,
see \cite{Kiselev-Ladyzhenskaya-1957}.
Of course, uniformity in $\nu$ is lost (or at least unknown) for incompressible fluids in 3d, and therefore,
one  can think of the effect of  the tangential component of the elastic stress, which is present only
when $\mu>0$, as a kind elastic of regularization.
Unsurprisingly then, in this paper the smallness restriction on the initial data
 is  sensitive to the size of $\mu>0$, and  this value will be fixed below.
  
There is an extensive literature on the existence of solutions to the equations of viscoelastic fluids,
see for example 
\cite{Cai-etal-2017}, \cite{Chemin-Masmoudi-2001}, \cite{Constantin-etal-2020},
 \cite{Elgindi-Liu-2015}, \cite{Elgindi-Rousset-2015},
\cite{Lei-Liu-Zhou-2007}, \cite{Lei-Liu-Zhou-2008}, \cite{Lin-Liu-Zhang-2005},
\cite{Lions-Masmoudi-2000}.

The proof of the global existence theorem is based on the energy method for symmetric
 systems using rotational and scaling vector fields.
The presence of the elastic stress term implies the local decay of energy inside a forward space-time
 cone of small aperture,
beginning with the analysis of the linearized problem in Proposition \ref{LocEn2Def}.  
The scaling operator is vital in this step.
Use of the rotational vector fields in  three space dimensions yields strong Sobolev
inequalities leading, in turn, to the basic interpolation estimates  described in Section \ref{BasicIneq}.
For small displacements, the decay of local energy and the interpolation inequalities 
can be combined to bootstrap the nonlinear terms to obtain the requisite {\em a priori} bounds for the
energy.  The precise null
structure of the nonlinear terms in the Hookean case is crucial in completing this portion of the argument,
see \eqref{I2'Bound}.

The challenge is to adapt this essentially hyperbolic method to a dissipative system.
The energy   provides control of the dissipative quantity $\nu\int_0^t\ltns{\nabla v(s)}ds$,
however, because the goal is to obtain estimates uniform in $\nu$, this 
 term, paradoxically,  is of little use.  In fact, dissipation  complicates
the local energy estimate by introducing extra terms and
requiring an additional   time integration in comparison to the case $\nu=0$.
The commutation properties of the scaling operator with the linearized operator
is also disrupted by the presence of dissipation, see \eqref{LinCommut},
making it necessary to  carefully monitor the use of the scaling operator.
Consequently, norms 
will take into account the total number of derivatives 
as well as the number of instances of the scaling operator.
The associated notation is explained in the next section, followed by a complete statement of the
main result, Theorem \ref{MainRsltElstic}.  The proof of the theorem will be presented in the remaining  sections.
\section{Notation}
Partial derivatives will be denoted by
\[
\nabla=(\partial_j)_{j=1}^3\isp{and}\partial=(\partial_t,\nabla).
\]
Define the antisymmetric
matrices
\begin{equation}
Z_1 = e_3 \otimes e_2 - e_2 \otimes e_3, \;
Z_2 = e_1 \otimes e_3 - e_3 \otimes e_1, \;
Z_3 = e_2 \otimes e_1 - e_1 \otimes e_2,
\end{equation}
where $\{e_1,e_2,e_3\}$ represents the standard basis in $\rr^3$. 
The rotational vector fields $\omt_i$ are  defined as the Lie derivative with respect to
the vector fields
\[
\Omega=x\wedge\nabla=(\Omega_\ell)_{\ell=1}^3=(\ip{Z_\ell x}{\nabla})_{\ell=1}^3.
\]
Thus we have
\begin{align}
& \omt_\ell \pi = \Omega_\ell \pi, \quad \text{for} \quad \pi :\rr^3\to \rr,\\
& \omt_\ell v = \Omega_\ell v - Z_\ell v, \quad \text{for} \quad v:\rr^3\to \rr^3, \\
& \omt_\ell F = \Omega_\ell F - [Z_\ell, F], \quad \text{for} \quad F :\rr^3\to \rr^{3\times3},
\end{align}
where $[\cdot ,\cdot]$ denotes the commutator of two matrices.
We will  rely heavily on the decomposition
\begin{equation}\label{GradDecomp}
\nabla = \omega \partial_r - \frac{\omega}{r} \wedge \Omega,\quad \omega=\frac{x}{r}.
\end{equation}

We also define the scaling operators
\[
S_0 = \ip{x}{\nabla}=r \partial_r\isp{and}
S = t\partial_t + r \partial_r.
\]
For a more concise notation we shall write
\begin{equation}
\Gamma = \{ \nabla, \omt \}.
\end{equation}
The scaling operators  are not included in $\Gamma$ because
 their occurrence will be tracked separately as is evident in
the following definition of the solution space:
\[
X^{p,q} = \left\{ U = (G,v): \rr^3 \rightarrow \rr^{3\times3} \times \rr^3 \; :\isum{p}\ltn{S_0^k\Gamma^a U(t)}
<\infty\right\}
\]
for integers $0 \leq q \leq p$. This is a Hilbert space with inner product
\begin{equation}
\ips{U_1}{U_2}{X^{p,q}} = \isum{p}\langle S_0^k\Gamma^a U_1,S_0^k\Gamma^a U_2\rangle_{L^2}.
\end{equation}
Thus, $p$ indicates the total number of derivatives taken, while $q$ indicates the number
of occurrences of $S_0$.  Here, $G=F-I$ plays the role of the displacement gradient, and $v$ represents the
velocity.  We do not include the pressure in the solution space because it will be expressed as a function
of $U=(G,v)$, see \eqref{PressureFormula}.

The energy associated with a solution $U = (F-I,v)=(G,v)$ of the PDEs \eqref{PDE1}, \eqref{PDE2} 
 is given by
\begin{equation}
\EE{p,q}=\isum{p}\left[\tfrac12\ltns{S^k\Gamma^a U(t)}+\nu\int_0^t\ltns{\nabla S^k\Gamma^a v(s)}\;ds\right].
\end{equation}
If $U(0) = U_0$, then the energy at time $t=0$ will be denoted as
\begin{equation}
\EEZ{p,q}\equiv
\EE[U][0]{p,q}=\frac12\ltns[X^{p,q}]{U_0}.
\end{equation}

We mention a few additional notational conventions.
We  write  $\jb{t}=(1+t^2)^{1/2}$.  
Summation on repeated indices is understood.  
Although all indices are lowered,  the first will be treated  contravariantly, and the rest will
act covariantly.
For vector-valued functions $v$,
$\nabla v$ is the matrix-valued function with entries $(\nabla v)_{ij}=\partial_jv_i$.
If $F$ is a matrix-valued function, then $\nabla\cdot F$ is the vector-valued function
with components $(\nabla\cdot F)_i=\partial_jF_{ij}$.  
We will use $F_{,j}$ to denote the $j^{\text{th}}$ column of $F$.
We write $A\lesssim B$ to mean that there exists a constant $C>0$ such that 
$A\le CB$.  All constants explicit or implied are independent of $0\le\nu\le1$.

\section{Main results}

Here now is a complete statement of the result to be proven.  
Since we shall consider small displacement gradients, we 
rewrite the system in terms of  $G=F-I$.
We shall return to the original system in Theorem \ref{OriginalPDEs} below.

\begin{theorem} \label{MainRsltElstic}
Fix $\mu=1$.  Choose  $(p,q)$ with $p\ge11$  and  $p\ge q> p^\ast$, where $p^\ast=\left[\frac{p+5}2\right]$.

There are  constants $C_0,\;C_1>1$, $0<\eps\ll1$   with the property  that
if  $v_0\in X^{p,q}$
satisfies
\begin{equation}
\label{IDHypoth}
 \nabla\cdot v_0=0\isp{and}C_0\|v_0\|_{X^{p^\ast,p^\ast}}^2(1+\|v_0\|_{X^{p,q}})<{\eps^2},
\end{equation}
then there is a unique pair $(G,v):[0,\infty)\times\rr^3\to\rr^{3\times 3}\times\rr^3$ such that
\begin{equation}
U=(G,v)\in C([0,\infty);X^{p,q})\cap\bigcap_{k=1}^{q}C^k((0,\infty),X^{p-k,q-k})
\end{equation}
and $(G,v)$ satisfies the PDEs 
%
\begin{align}
\label{Pde1}
& \partial_t G - \nabla v = \nabla v G - v \cdot \nabla G \\
\label{Pde2}
&\partial_t v - \nabla \cdot G - \nu \Delta v = \nabla \cdot (GG^T) - v \cdot \nabla v - \nabla \pi
\end{align}
with pressure
\begin{equation}
\label{PressureFormula}
\pi=\Delta^{-1}\partial_i\partial_j\left[G_{ik}G_{jk}-v_iv_j\right]
\end{equation}
%
%
and initial data $(0,v_0)$.

Moreover, $U=(G,v)$ satisfies the estimates
\begin{equation}
\label{Apriori}
\EE{p,q}\le C_1\EEZ{p,q}\jb{t}^{C_1{\eps}}\isp{and}
\EE{p^\ast,p^\ast}<{\eps^2},
\end{equation}
for $ t\in[0,\infty)$.
\end{theorem}

\begin{proof}[Outline of Proof]
The strategy is the same as was used for the scalar model problem studied in \cite{Jonov-Sideris-2015}.
The system \eqref{Pde1},\eqref{Pde2} is locally well-posed in $X^{p,q}$.
The norm in $X^{p,q}$ can be controlled by an expression involving $\EE{p,q}$, and
thus, global existence follows by showing that $\EE{p,q}$ remains finite.
Here, we shall focus only on establishing the {\em a priori} estimates \eqref{Apriori}. 
The proof of these estimates will be spread across the remaining sections.
\end{proof}

\begin{corollary}
\label{ConstrCorollary}
The solution pair $(G,v)$ given in Theorem \ref{MainRsltElstic} satisfies the constraints
%
\begin{align}
\label{Constr1}
\quad&\nabla \cdot v = 0\\
\label{Constr2}
& \nabla \cdot G^T = 0 \\
\label{Constr3}
& \partial_k G_{ij} - \partial_j G_{ik} = G_{\ell j} \partial_\ell G_{ik} - G_{\ell k} \partial_\ell G_{ij}
 \equiv Q_{ijk}(G,\nabla G).
\end{align}
%
%
\end{corollary}

\begin{proof}
Let $\delta=\nabla\cdot v$ and $\xi=\nabla\cdot G^\top$.  Then \eqref{Pde1},\eqref{Pde2},\eqref{PressureFormula}
imply that
\begin{align}
&\partial_t\xi-\nabla\delta=G^\top\nabla\delta-v\cdot\nabla\xi\\
&\partial_t\delta-\nabla\cdot\xi=v\cdot\nabla\delta+\delta^2.
\end{align}
Since $(\xi,\delta)$ vanishes initially, \eqref{Constr1},\eqref{Constr2} follow by uniqueness.

We will prove \eqref{Constr3} after the next result.
\end{proof}

We now connect Theorem \ref{MainRsltElstic} to the original problem.

\begin{theorem}
\label{OriginalPDEs}
If $(G,v)$ is the solution pair given in Theorem \ref{MainRsltElstic},
then $F=I+G$ is the velocity gradient of the flow determined by $v$ expressed in spatial coordinates,
$F:[0,\infty)\times \rr^3\to\slt$, 
and $(F,v)$ satisfies \eqref{PDE1},\eqref{PDE2},\eqref{PDE3}.

\end{theorem}

\begin{proof}
$(F,v)$ satisfies \eqref{PDE1},\eqref{PDE2},\eqref{PDE3} by \eqref{Pde1},\eqref{Pde2},\eqref{Constr1},\eqref{Constr2}.

Since $p>3$, we have $X^{p,q}\subset H^3$, and thus,
the velocity vector field $v\in C([0,\infty),X^{p,q})$ is a bounded
$C^1$ function on $\rr^3$ for each $t\in[0,\infty)$, by the Sobolev lemma,
and it is continuous in $t$.  As such it defines a flow on $\rr^3$ through
\[
D_tx(t,y)=v(t,x(t,y)),\quad x(0,y)=y.
\]

By \eqref{Constr1}, $\nabla\cdot v=0$, and so
$x(t,\cdot)$ is a one-parameter family of volume preserving deformations of $\rr^3$.
The  deformation gradient $\bar F_{ij}(t,y)=(D_yx)_{ij}(t,y)=D_{y_j}x_i(t,y)$ is the unique solution of the system
\begin{equation}
\label{vareq}
D_t\bar F(t,y)=\nabla v(t,x(t,y))\bar F(t,y),\quad \bar F(0,y)=I.
\end{equation}
Since $x(t,\cdot)$ is volume preserving, we see that $\bar F:[0,\infty)\times\rr^3\to\slt$.

Since $F$ solves \eqref{PDE1}, $F(t,x(t,y)) $ solves the initial value problem \eqref{vareq}, and so,
$F(t,x(t,y))=\bar F(t,y)$.

The inverse  $y(t,\cdot)$ of the deformation $x(t,\cdot)$ is called the reference or back-to-labels map, and
it takes  spatial points $x$ to material points $y$.  
Using this, we see that $F(t,x)=\bar F(t,y(t,x))$, which shows that $F$ is the deformation gradient in
spatial coordinates.
It follows that  $F:[0,\infty)\times\rr^3\to\slt$.

\end{proof}

\begin{proof}[Proof of \eqref{Constr3}]
Referring to the previous paragraphs, the deformation gradient satisfies
the relation $D_{y_j}\bar F_{ik}(t,y)=D_{y_k}\bar F_{ij}(t,y)$.
Switching to spatial coordinates, this is equivalent to the compatibility condition
$F_{\ell k}\partial_\ell F_{ij}=F_{\ell j}\partial_\ell F_{ik}$.  Since $F=I+G$, this, in turn, implies \eqref{Constr3}.

\end{proof}

\section{Commutation}
\label{CommutationSec}
Using the following notation:
\begin{equation}
U = (G,v), \quad \text{where} \quad G \in \rr^{3\times3}, \; v \in \rr^3,
\end{equation}
we  define
\begin{equation}
A(\nabla) U = (\nabla v, \nabla \cdot G),\quad BU=(0,v), 
\end{equation}
and the linearized ovperator
\begin{equation}
LU=(I\partial_t-A(\nabla)-\nu B\Delta)U.
\end{equation}
We can then rewrite \eqref{Pde1} and  \eqref{Pde2} as
%
\begin{equation} \label{PdeMatrx}
LU= N(U,\nabla U) + (0, - \nabla \pi),
\end{equation}
where the nonlinearity is of the form
\begin{equation}\label{DefN}
N(U,\nabla U) = ( N_1(U,\nabla U), N_2(U,\nabla U) )
\end{equation}
with
\begin{equation}\label{DefN1N2}
\begin{aligned}
& N_1(U,\nabla U) = \nabla v G - v \cdot \nabla G, \\
& N_2(U,\nabla U) = \nabla \cdot (GG^T) - v \cdot \nabla v.
\end{aligned}
\end{equation}

\begin{lemma}
\label{CommutLemma}
Let $p\ge3$ and $p\ge q\ge1$.
Suppose that 
\[
(G,v,\pi )\in C([0,\infty);X^{p,q})\cap\bigcap_{k=1}^{q}C^k((0,\infty),X^{p-k,q-k}).
\]
Then for  any  integer $k$ and multi-index $a$ such that $0\le k\le q$ and $|a|+k\le p$, there holds
\begin{align} \label{LinCommut}
&L S^k \Gamma^a U = (S+1)^k \Gamma^a L U +\nu\sum_{j=0}^{k-1} (-1)^{k-j} {k \choose j}  B \Delta S^j \Gamma^a U\\
\label{NonLinCommut}
&(S+1)^k \Gamma^a N(U,\nabla U)\\
&\ \hspace{.5in} =\sum_{\stacktwo{a_1+a_2=a}{k_1+k_2=k}}
\frac{a!}{a_1! \; a_2!} \frac{k!}{k_1! \; k_2!} N(S^{k_1} \Gamma^{a_1} U, \nabla S^{k_2} \Gamma^{a_2} U)\\
\label{PressureCommut}
&(S+1)^k \Gamma^a \nabla \pi=\nabla S^k\Gamma^a\pi.
\end{align}

\end{lemma}

We remark that when $k=0$ the sum is empty in \eqref{LinCommut}.
\begin{proof}
The linear operator $L$ defined in \eqref{PdeMatrx} is translationally and rotationally invariant
which implies that
\begin{equation}
L \Gamma^a U = \Gamma^a L U
\end{equation}
The scaling operator $S$, however, does not commute with $L$.
Since
\begin{equation}\label{scalarCommut}
\begin{aligned}
&\partial S  = (S+1) \partial  \\
&\Delta S  = (S+2) \Delta ,
\end{aligned}
\end{equation}
we have for any $k>0$
\begin{align}
L S^k & = (S+1)^k (\partial_t-A(\nabla))-\nu B(S+2)^k\Delta\\
&=(S+1)^k L+\nu B[(S+1)^k-(S+2)^k]\Delta.
\end{align}
We  complete the proof of \eqref{LinCommut} with
\begin{align*}
(S+1)^k\Delta  &= [(S+2)-1]^k\Delta \\
&=\sum_{j=0}^k{k\choose j}(S+2)^j(-1)^{k-j}\Delta \\
&=(S+2)^k\Delta +\sum_{j=0}^{k-1}(-1)^{k-j}{k\choose j}\Delta S^j .
\end{align*}

The statement \eqref{NonLinCommut} for the nonlinear terms is also  a consequence of 
the Leibnitz-type formulas
\[
\Gamma N_j(U,\nabla U)=N_j(\Gamma U,\nabla U)+N_j(U,\nabla\Gamma  U),
\]
\[
(S+1)N_j(U,\nabla U)=N_j(SU,\nabla U)+N_j(U,\nabla SU),\quad j=1,2.
\]

The statement \eqref{PressureCommut} is also easily verified.
\end{proof}

The appearance of the additional summand on the right-hand side of \eqref{LinCommut} will
force us to proceed by induction on $k$.

\begin{lemma}
Let $p\ge3$ and $p\ge q\ge1$.
Suppose that 
\[
(G,v)\in C([0,\infty);X^{p,q})\cap\bigcap_{k=1}^{q}C^k((0,\infty),X^{p-k,q-k})
\]
satisfy the constraints 
\eqref{Constr1},\eqref{Constr2},\eqref{Constr3}.  
Then for  any  integer $m$ and  multi-index $a$ such that $0\le m\le q$ and $|a|+m\le p$, there holds

%
\begin{align}
\label{ConstrVectrFld1}
\quad&\nabla \cdot S^m \Gamma^a v = 0\\
\label{ConstrVectrFld2}
& \nabla \cdot (S^m \Gamma^aG)^T = 0 \\
\label{ConstrVectrFld3}
& \partial_k (S^m \Gamma^a G)_{ij} - \partial_j (S^m \Gamma^a G)_{ik} \\
& \hs{1} = \sum_{\stacktwo{a_1+a_2=a}{m_1+m_2=m}}
\frac{a!}{a_1! \; a_2!} \frac{m!}{m_1! \; m_2!} Q_{ijk}(S^{m_1} \Gamma^{a_1} G, \nabla S^{m_2} \Gamma^{a_2} G) \\
& \hs{1} \equiv \td{Q}_{ijk} (G, \nabla G).
\end{align}

\end{lemma}

\begin{proof}
These statements follow by applying $(S+1)^m\Gamma^a$ to the equations \eqref{Constr1},\eqref{Constr2},\eqref{Constr3}
and commuting with $\nabla$.
Our convention for computing $\omt_\ell Q$ is
\begin{equation}
\label{QConv}
(\omt_\ell Q)_{ijk}=\Omega Q_{ijk}-(Z_\ell)_{iI}Q_{Ijk}+Q_{iJk}(Z_\ell)_{Jj}+Q_{ijK}(Z_\ell)_{Kk}.
\end{equation}
\end{proof}

\begin{lemma}\label{PressLemma}
Let $p\ge3$ and $p\ge q\ge1$.
Suppose that 
\[
(G,v)\in C([0,\infty);X^{p,q})\cap\bigcap_{k=1}^{q}C^k((0,\infty),X^{p-k,q-k}).
\]
If $\pi$ is defined by \eqref{PressureFormula}, then
for  any  integer $k$ and  multi-index $a$ such that $0\le k\le q$ and $|a|+k\le p-1$,
there holds
\begin{equation}
   S^k \Gamma^a \pi=
 \Delta^{-1}\partial_i\partial_jS^k \Gamma^a \left[(G G^T)_{ij} - v_iv_j \right].
\end{equation}
\end{lemma}
\begin{proof}

Write
\[
\Delta\pi=\partial_i\partial_j\left[(G G^T)_{ij} - v_iv_j \right].
\]
It is straight-forward to check that
\[
\Delta S\pi=(S+2)\Delta\pi=\partial_i\partial_jS\left[(G G^T)_{ij} - v_iv_j \right]
\]
and
\[
\Delta\omt_k\pi=\omt_k\Delta\pi=\partial_i\partial_j\omt_k\left[(G G^T)_{ij} - v_iv_j \right].
\]
It follows that 
\[
\Delta S^k\Gamma^a\pi=\partial_i\partial_jS^k\Gamma^a\left[(G G^T)_{ij} - v_iv_j \right].
\]

\end{proof}

\section{Basic Inequalities}
\label{BasicIneq}

The   results in this section will be used to estimate the nonlinear terms.
This is the only  place where the dimension $n=3$ enters the argument.

We introduce cut-off functions to distinguish two time-space regions,
referred to as interior and exterior.
Define
%

\begin{equation} \label{CutOffs}
\zeta(t,x)=\psi\left(\frac{|x|}{\sigma\jb{t}}\right)
\quad\text{and}\quad
\eta(t,x)=1-\psi\left(\frac{2|x|}{\sigma\jb{t}}\right),
\end{equation}
where
\begin{equation}\label{PsiCutOff}
\psi\in C^\infty(\rr),\quad
\psi(s)=
\begin{cases}
1,& s\le1/2\\
0,&s\ge1
\end{cases},
\quad \psi'\le0,\isp{and}\jb{t}=(1+t^2)^{1/2}.
\end{equation}
The parameter $\sigma>0$ in the definition of $\zeta$ and $\eta$ will be  chosen later to be sufficiently small.
Note that this is not a partition of unity.
The cut-off functions satisfy the following inequalities:
\begin{equation}\label{Unity}
1\le \zeta+\eta
\quad\text{and}\quad
1-\eta\le\zeta^2
\end{equation}
and
\begin{equation}
\label{WeightDer}
\jb{r+t}\Big[|\partial_{t,x}\zeta (t,x)|+|\partial_{t,x}\eta(t,x)|\Big]\lesssim 1.
\end{equation}


We define the following localized energy:
\begin{equation}
\label{LocEnDef}
\YYI{p,q}=\isum{p-1} \ltns{\zeta\nabla S^k\Gamma^a U(t)},\quad q<p.
\end{equation}
This quantity will  be estimated in Section \ref{DecayEstimates}.

The next three Lemmas were proven in \cite{Jonov-Sideris-2015},
(with the slight notational difference that the quantity $\mathcal Y_{p,q}$ is denoted by $\mathcal Y_{p,q}^{\text{int}}$.)
\begin{lemma}[Proposition 6.2, \cite{Jonov-Sideris-2015}]
Suppose that $U=(G,v):[0,T)\times\rr^3\to \rr^{3\times3} \times \rr^3$ satisfies
\eqref{Constr1},\eqref{Constr2} and
\begin{equation}
\YYI{3,0}+\EE{2,0}<\infty.
\end{equation}
Then using the weights \eqref{CutOffs}, we have
\begin{align}
\label{SobInt1}
\quad
&\ltn[L^\infty]{\;\zeta\; U(t)\;}
\ \lesssim \ \YYIH{2,0} + \jb{t}^{-1} \EEH{1,0}\\
\label{SobInt2}
&\ltn[L^\infty]{\;r\zeta\;\nabla U(t)\;}
\ \lesssim \ \YYIH{3,0}+ \jb{t}^{-1} \EEH{2,0}\\
\label{SobHrdy}
&\ltn{\;r^{-1}\zeta \; U(t)\;}
\ \lesssim \
\YYIH{1,0}+\jb{t}^{-1}\EEH{0,0}.
\end{align}

\end{lemma}

\begin{lemma}[Lemma 7.1, \cite{Jonov-Sideris-2015}]
\label{IntCalcIneq}
Suppose that $U:[0,T)\times\rr^3\to\rr^{3\times3} \times \rr^3$.
If
\begin{equation}
 k_1+k_2+  |a_1| + |a_2| \leq \pbar
 \quad\text{ and  }\quad
 k_1 + k_2 \leq  \qbar,
\end{equation}
  then we have
\begin{multline}
\ltn{ \; \zeta  | S^{k_1} \Gamma^{a_1} U(t) |  \;  | S^{k_2} \Gamma^{a_2+1} U(t) |\;  } \\
 \lesssim \  \left (   \YYIH{\left[  \frac{\pbar +5 }{2}  \right] , {\left[ \frac{\pbar }{2}\right]}}      +
    \jb{t}^{-1} \EEH{ \left[  \frac{\pbar +3 }{2}  \right] , {\left[ \frac{\pbar }{2}\right]  } }\right )
\EEH{\pbar+1,\qbar},
\end{multline}
provided the right-hand side is finite.

In the special case when $k_2+|a_2| < \pbar$, we have
\begin{multline}
\ltn{ \; \zeta   | S^{k_1} \Gamma^{a_1} U(t) |  \;  | S^{k_2} \Gamma^{a_2+1} U(t) |\;  } \\
\lesssim \  \left (   \YYIH{\left[  \frac{\pbar +5 }{2}  \right] ,  {\left[ \frac{\pbar }{2}\right]  }   }
+    \jb{t}^{-1} \EEH{ \left[  \frac{\pbar +3 }{2}  \right] ,{\left[ \frac{\pbar }{2}\right]  } }   \right )
\EEH{\pbar,\qbar},
\end{multline}
provided the right-hand side is finite.
\end{lemma}

\begin{lemma}[Lemma 7.2, \cite{Jonov-Sideris-2015}]
\label{ExtCalcIneq}
Suppose that $U:[0,T)\times\rr^3\to \rr^{3\times3} \times \rr^3$.
If
\begin{equation}
k_1+k_2+  |a_1| + |a_2| \leq \pbar
\quad \text{ and  }\quad
k_1+k_2\leq  \qbar,
\end{equation}
 then we have
\begin{equation}
\ltn{ \; \eta   |S^{k_1} \Gamma^{a_1}  U(t) |  \;   |S^{k_2} \Gamma^{a_2+1} U(t)| \;}
\lesssim \ \jb{t}^{-1} \EEH{ \left[  \frac{\pbar +5 }{2}  \right] ,  {\left[ \frac{\pbar }{2}\right]  } }
\EEH{\pbar+1,\qbar},
\end{equation}
provided the right-hand side is finite.

In the special case when $k_2+|a_2| < \pbar$, we have
\begin{equation}
\ltn{ \; \eta   |S^{k_1} \Gamma^{a_1}  U(t) |  \;   |S^{k_2} \Gamma^{a_2+1} U(t)| \;}
\lesssim \ \jb{t}^{-1} \EEH{ \left[  \frac{\pbar +5 }{2}  \right] ,  {\left[ \frac{\pbar }{2}\right]  }}
\EEH{\pbar,\qbar},
\end{equation}
provided the right-hand side is finite.
\end{lemma}

This next result takes into account the constraints, and it will play a key role
in assessing nonlinear interactions in \eqref{I2'Bound}.

\begin{prop}
\label{SpecialSob}
Suppose that $U=(G,v):[0,T)\times\rr^3\to \rr^{3\times3} \times \rr^3$ satisfies
\eqref{Constr1},\eqref{Constr2} and $\EE{2,0}<\infty$, for all $t\in[0,T)$.
Then
\begin{equation}
\label{SobOmega}
\ltn[L^\infty]{\;\eta\; G(t)^\top\omega}+\ltn[L^\infty]{\;\eta\;\omega\cdot v(t)}
\lesssim 
\jb{t}^{-3/2} \EEH{2,0}.
\end{equation}
\end{prop}

\begin{proof}
Taking $U=(G,v)$ and $\omega=x/|x|$,  define 
\[
\omega\otimes\omega \;U(t,x)= (\omega\otimes\omega\; G(t,x), \omega\otimes\omega \;v(t,x)).
\]
Note that 
\[
\omt \;\omega= 0,
\]
and so 
\[
\omt\;[\;\omega\otimes\omega \;U\;] = \omega\otimes\omega\; \omt U.
\]

The key point is that, from the gradient decomposition \eqref{GradDecomp} and the constraints \eqref{ConstrVectrFld1}, \eqref{ConstrVectrFld2}, 
we have
\begin{multline}
\label{NullConstr}
\sum_{|a|\le1}|\omega\otimes\omega \;\partial_r\omt^aU|\\
\lesssim
\sum_{|a|\le1}\left[|(\partial_r\omt^aG)^\top\omega|+|\omega\cdot\partial_r\omt^av|\right]
\lesssim 
\frac1r\sum_{|a|\le2}|\omt^aU|.
\end{multline}

We appeal to the following inequality
\begin{equation}
\ltn[L^\infty]{rU} 
\lesssim \ \left(\sum_{|a|\le1}\ltn{\partial_r\omt^a U}\sum_{|a|\le 2}\ltn{\omt^aU}\right)^{1/2}
\end{equation}
which follows from Lemma 3.3 (3.14b) of \cite{Sideris-2000}.
Apply this to
$\eta\; \omega\otimes\omega\; U$.
Since $\jb{t}\lesssim r$ on the support of $\eta$, we have  
\begin{align}
\jb{t}^2&\left[\;\ltns[L^\infty]{\;\eta\; G(t)^\top\omega}+\ltns[L^\infty]{\;\eta\;\omega\cdot v(t)}\right]\\
&\lesssim \ltns[L^\infty]{\;r\;\eta\; \omega\otimes\omega\;U}\\
&\lesssim \sum_{|a|\le1}\ltn[L^2]{\partial_r\omt^a [\eta\; \omega\otimes\omega\; U]}
\sum_{|a|\le 2}\ltn[L^2]{\omt^a[\eta\; \omega\otimes\omega\; U]}\\
&\lesssim \sum_{|a|\le1}\ltn[L^2]{\omega\otimes\omega\;\partial_r [\eta\;  \omt^aU]}
\sum_{|a|\le 2}\ltn[L^2]{\omt^a U}.
\end{align}
Using \eqref{WeightDer} and \eqref{NullConstr}, this is bounded by
\begin{multline}
 \sum_{|a|\le1}\left[\ltn[L^2]{\omega\otimes\omega\;\eta\;\partial_r   \omt^aU}
+\jb{t}^{-1}\ltn{\omt^aU}\right]
\sum_{|a|\le 2}\ltn[L^2]{\omt^a U}\\
\lesssim\jb{t}^{-1}
\sum_{|a|\le 2}\ltns[L^2]{\omt^a U}
\lesssim\jb{t}^{-1}\EE{2,0}.
\end{multline}
This shows that
\begin{equation}
\ltns[L^\infty]{\;\eta\; G(t)^\top\omega}+\ltns[L^\infty]{\;\eta\;\omega\cdot v(t)}
\lesssim \jb{t}^{-3}\EE{2,0},
\end{equation}
from which the Proposition follows.
\end{proof}

Following our notational conventions, we may  write $(G^\top\omega)_j=\omega\cdot G_{,j}$.
\section{Estimates for the Linearized System }

In this section we provide estimates for solutions of the linearized system
%
\begin{align}
\label{Pde1Lin}
& \partial_t G - \nabla v = H \\
\label{Pde2Lin}
&\partial_t v - \nabla \cdot G - \nu \Delta v = h\\
\label{Constr3Lin}
&\partial_k G_{ij} - \partial_j G_{ik}   =   Q_{ijk}\quad i,j,k=1,2,3.
\end{align}

%
\begin{lemma}
\label{IntDecayLemma}
Assume that $\sigma$ in \eqref{CutOffs} is sufficiently small and that $\nu\le1$.
Assume that the functions $H$, $\nabla \cdot H$, $Q$, and $h$ all belong to the space
$ L^2([0,T];L^2(\rr^3))$, for some $0<T<\infty$.

If $U=(G,v)$ is a solution of \eqref{Pde1Lin},\eqref{Pde2Lin},\eqref{Constr3Lin} such that
\begin{equation}
\sup_{0\le t\le T}\EE[U]{1,1}<\infty,
\end{equation}
then
for any  $0\le\theta\le1$,
\begin{multline}
\int_0^T \jb{t}^{\theta} \left[   \ltns{\zeta \nabla  U} +
\nu^2 \ltns{\zeta  \Delta v} \right] dt \\
\lesssim  \nu \jb{T}^{\theta-2} \EE[U][T]{1,0}
  + \int_0^T \jb{t}^{\theta-2} \EE [U]{1,1} dt  \\
  +  \int_0^T \jb{t}^{\theta} \left[
  2 \nu  \ltip{\zeta  \nabla \cdot G}{\zeta  \nabla \cdot H} \right.  \\
  + \left.   \ltns{\zeta Q} + \ltns{\zeta H} + \ltns{\zeta  h} \right] dt.
\end{multline}
\end{lemma}

\begin{proof}
Multiplying \eqref{Pde1Lin} and \eqref{Pde2Lin} by $t$ and using that $S=t\partial_t + r \partial_r$, we have that
\begin{align}
& t \nabla v = -r\partial_r G + S G - t H \\
& t \nabla \cdot G + t \nu \Delta v = -r\partial_r v + S v - t h.
\end{align}
Next, we multiply each equation by $\zeta$ and take the $L^2$-norm
\begin{align}
& \ltns{\zeta t \nabla v}
\leq \ltns{\zeta r\partial_r G} + \ltns{\zeta S G} + \ltns{\zeta t H} \\
& \ltns{\zeta t \nabla \cdot G} + \nu^2 \ltns{\zeta t \Delta v}
+ 2 \ltip{\zeta t \nabla \cdot G}{\zeta t \nu \Delta v} \\
& \hspace{5cm}
\leq \ltns{\zeta r\partial_r v} + \ltns{\zeta S v} + \ltns{\zeta t h}.
\end{align}
Adding the two inequalities, we obtain
\begin{multline} \label{LETemp1}
t^2 \left[ \ltns{\zeta \nabla \cdot G} +
2 \nu \ltip{\zeta  \nabla \cdot G}{\zeta \Delta v}  +
\nu^2 \ltns{\zeta  \Delta v} + \ltns{ \zeta \nabla v} \right] \\
\leq \ltns{\zeta r\partial_r U} + \EE [U]{1,1} +
t^2 (\ltns{\zeta H} + \ltns{\zeta  h}),
\end{multline}
where $U = (G,v)$.

Taking the divergence of \eqref{Pde1Lin}, we have
\begin{equation}
\Delta v = \partial_t \nabla \cdot G - \nabla \cdot H
\end{equation}
and so we can write the inner product as
\begin{equation}\label{LETemp2}
2 \nu \ltip{\zeta  \nabla \cdot G}{\zeta \Delta v}
\; = \; 2 \nu \ltip{\zeta  \nabla \cdot G}{\zeta \partial_t \nabla \cdot G}
- 2 \nu \ltip{\zeta  \nabla \cdot G}{\zeta  \nabla \cdot H}.
\end{equation}
The first term can be bounded below as follows:
\begin{align} \label{LETemp3}
& 2 \nu \ltip{\zeta  \nabla \cdot G}{\zeta \partial_t \nabla \cdot G} \\
& \hspace{2cm}
= \nu \int_{\rr^3} \zeta^2 \partial_t |\nabla \cdot G |^2 dx \\
& \hspace{2cm}
\geq \nu \partial_t \ltns{\zeta \nabla \cdot G } - C \nu \int_{\rr^3} \zeta \jb{t}^{-1} |\nabla \cdot G|^2 dx \\
& \hspace{2cm}
\geq \nu \partial_t \ltns{\zeta \nabla \cdot G } - \frac12 \ltns{\zeta \nabla \cdot G }\\
&\hs{3}
-2 C^2 \nu^2 \jb{t}^{-2} \ltns{ \nabla \cdot G },
\end{align}
where we have used Young's inequality and the fact that by \eqref{WeightDer},
$\partial_t \zeta^2 \leq C \zeta \jb{t}^{-1}$,
for some constant $C$. Inserting \eqref{LETemp2} and \eqref{LETemp3} into \eqref{LETemp1} and using
 $\ltns{ \nabla \cdot G(t) }\lesssim \EE [U]{1,1}$, we have
\begin{multline} \label{LETemp4}
t^2 \left[  \nu \partial_t \ltns{\zeta \nabla \cdot G} +  \frac12 \ltns{\zeta \nabla \cdot G} +
\nu^2 \ltns{\zeta  \Delta v} + \ltns{ \zeta \nabla v} \right] \\
\lesssim \ltns{\zeta r\partial_r U} + \EE [U]{1,1} +
2 \nu t^2 \ltip{\zeta  \nabla \cdot G}{\zeta  \nabla \cdot H} \\
+ t^2 (\ltns{\zeta H} + \ltns{\zeta  h}).
\end{multline}

Choosing $0 \leq \theta \leq 1$, we multiply \eqref{LETemp4} by $\jb{t}^{\theta-2}$, and then
we integrate
\begin{multline}\label{LETemp4.1}
\int_0^T t^2 \jb{t}^{\theta-2} \left[  \tfrac12 \ltns{\zeta \nabla \cdot G} +
\nu^2 \ltns{\zeta  \Delta v} + \ltns{ \zeta \nabla v} \right] dt \\
\lesssim \int_0^T \jb{t}^{\theta-2} \left[
\ltns{\zeta r\partial_r U} + \EE [U]{1,1} +
2 \nu t^2 \ltip{\zeta  \nabla \cdot G}{\zeta  \nabla \cdot H} \right. \\
+ \left. t^2 (\ltns{\zeta H} + \ltns{\zeta  h}) \right] dt \\
- \int_0^T t^2 \jb{t}^{\theta-2} \nu \partial_t \ltns{\zeta \nabla \cdot G} dt.
\end{multline}

Next, we estimate the time-derivative term on the right hand side. Since
\begin{equation}
\partial_t ( \nu t^2 \jb{t}^{\theta-2}) =\nu t \jb{t}^{\theta-4}(2+\theta t^2)\le 2\nu t \jb{t}^{\theta-2}\leq \tfrac14 t^2 \jb{t}^{\theta-2} +
C \nu^2 \jb{t}^{\theta-2},
\end{equation}
integration by parts yields
\begin{align}
- \int_0^T t^2 \jb{t}^{\theta-2}\; \nu \partial_t \ltns{ &\zeta \nabla \cdot G} dt \\
& \leq \int_0^T \left( \tfrac14 t^2 \jb{t}^{\theta-2} + C \nu^2 \jb{t}^{\theta-2} \right)
\ltns{\zeta \nabla \cdot G} dt \\
& \leq \tfrac14 \int_0^T  t^2 \jb{t}^{\theta-2} \ltns{\zeta \nabla \cdot G} dt
+ C \int_0^T  \jb{t}^{\theta-2} \EE[U]{1,0}dt.
\end{align}
Substitution in \eqref{LETemp4.1} gives
\begin{multline}
\int_0^T t^2 \jb{t}^{\theta-2} \left[  \tfrac14 \ltns{\zeta \nabla \cdot G} +
\nu^2 \ltns{\zeta  \Delta v} + \ltns{ \zeta \nabla v} \right] dt \\
\lesssim \int_0^T \jb{t}^{\theta-2} \left[
\ltns{\zeta r\partial_r U} + \EE [U]{1,1} +
2 \nu t^2 \ltip{\zeta  \nabla \cdot G}{\zeta  \nabla \cdot H} \right. \\
+ \left. t^2 (\ltns{\zeta H} + \ltns{\zeta  h}) \right] dt.
\end{multline}
By Lemma \ref{Div2Grad} (below), we can relpace the divergence term on the left-hand side by the full gradient
\begin{multline}
\int_0^T t^2 \jb{t}^{\theta-2} \left[   \ltns{\zeta \nabla  U} +
\nu^2 \ltns{\zeta  \Delta v} \right] dt \\
\lesssim \int_0^T \jb{t}^{\theta-2} \left[
\ltns{\zeta r\partial_r U} + \EE [U]{1,1} +
2 \nu t^2 \ltip{\zeta  \nabla \cdot G}{\zeta  \nabla \cdot H} \right. \\
+ \left. t^2 ( \ltns{\zeta Q} + \ltns{\zeta H} + \ltns{\zeta  h}) \right] dt.
\end{multline}
Using that $t^2 = \jb{t}^2-1$ and $\ltns{\zeta \nabla  U(t)}\lesssim\EE [U]{1,1}$, we can write
\begin{multline}
\int_0^T \jb{t}^{\theta} \left[   \ltns{\zeta \nabla  U} +
\nu^2 \ltns{\zeta  \Delta v} \right] dt \\
\lesssim \int_0^T \jb{t}^{\theta-2} \left[
\ltns{\zeta r\partial_r U}  +  \EE [U]{1,1}  +  \nu^2 \ltns{\zeta  \Delta v}  \right.\\
+ \left. 2 \nu t^2 \ltip{\zeta  \nabla \cdot G}{\zeta  \nabla \cdot H} \right. \\
+ \left. t^2 ( \ltns{\zeta Q} + \ltns{\zeta H} + \ltns{\zeta  h}) \right] dt.
\end{multline}
Since $r \leq \sigma \jb{t}$ on the support of $\zeta$, we have the following estimate:
\begin{equation}
\int_0^T \jb{t}^{\theta-2} \ltns{\zeta r\partial_r U} dt
    \;\leq\; \int_0^T \sigma^2 \jb{t}^{\theta} \ltns{\zeta \partial_r U} dt.
\end{equation}
For small enough $\sigma$, the last term can be absorbed on the left-hand side of the main
inequality and thus, we obtain
\begin{multline} \label{LETemp5}
\int_0^T \jb{t}^{\theta} \left[   \ltns{\zeta \nabla  U} +
\nu^2 \ltns{\zeta  \Delta v} \right] dt \\
\hspace{1cm}  \lesssim \int_0^T \jb{t}^{\theta-2} \left[
\EE [U]{1,1}  +  \nu^2 \ltns{\zeta  \Delta v}
+ \left. 2 \nu t^2 \ltip{\zeta  \nabla \cdot G}{\zeta  \nabla \cdot H}  \right.\right.\\
\hspace{2cm}   + \left. t^2 ( \ltns{\zeta Q} + \ltns{\zeta H} + \ltns{\zeta  h}) \right] dt.
\end{multline}
The Laplacian term on the right has the following bound:
\begin{align}
\label{Laplace}
\int_0^T &\jb{t}^{\theta-2} \nu^2 \ltns{\zeta  \Delta v}dt
\; = \; \int_0^T \jb{t}^{\theta-2} \nu^2 \frac{d}{dt} \int_0^t \ltns{\zeta  \Delta v}dsdt \\
&  =\jb{T}^{\theta-2} \nu^2 \int_0^T \ltns{\zeta  \Delta v}dt
\;+\; \int_0^T (2-\theta)t \jb{t}^{\theta-4} \nu^2  \int_0^t \ltns{\zeta  \Delta v}dsdt \\
&  \lesssim   \nu \jb{T}^{\theta-2} \EE[U][T]{1,0}
   \;+\; \nu \int_0^T \jb{t}^{\theta-2} \EE[U][t]{1,0}dt.
\end{align}
Substituting into \eqref{LETemp5} and using that $\nu \leq 1$, we arrive at
\begin{multline}
\int_0^T \jb{t}^{\theta} \left[   \ltns{\zeta \nabla  U} +
\nu^2 \ltns{\zeta  \Delta v} \right] dt \\
\lesssim  \nu \jb{T}^{\theta-2} \EE[U][T]{1,0}
+ \int_0^T \jb{t}^{\theta-2} \EE [U]{1,1} dt  \\
+  \int_0^T \jb{t}^{\theta} \left[
  2 \nu  \ltip{\zeta  \nabla \cdot G}{\zeta  \nabla \cdot H} \right.  \\
 + \left.   \ltns{\zeta Q} + \ltns{\zeta H} + \ltns{\zeta  h} \right] dt.
\end{multline}
\end{proof}

In the proof of Lemma \eqref{IntDecayLemma}, we used the following estimate:
\begin{lemma}
\label{Div2Grad}
If $G\in H^1(\rr^3,\rr^{3\times3})$ and
\begin{equation} \label{DerIndxCommt}
\partial_k G_{ij} - \partial_j G_{ik} = Q_{ijk}
\end{equation}
for $i,j,k = 1,2,3$ and $\ltn{Q} < \infty$,
then
\begin{equation}
\frac12 \ltns{\zeta \nabla G} - \ltns{\zeta \nabla \cdot G}
\lesssim  \jb{t}^{-2} \ltns{ G} +  \ltns{\zeta Q}.
\end{equation}
\end{lemma}
\begin{proof}
The constraint \eqref{DerIndxCommt} implies that
\begin{align}
- |\nabla \cdot G|^2 &= - \partial_j G_{ij}\partial_k G_{ik} \\
& = - \partial_k (\partial_j G_{ij} G_{ik}) + \partial_j \partial_k G_{ij} G_{ik} \\
& = - \partial_k (\partial_j G_{ij} G_{ik}) + \partial_j (\partial_j G_{ik} + Q_{ijk}) G_{ik} \\
& = - \partial_k (\partial_j G_{ij} G_{ik}) + \partial_j \partial_j G_{ik} G_{ik} + \partial_j Q_{ijk} G_{ik} \\
& = - \partial_k (\partial_j G_{ij} G_{ik}) + \partial_j (\partial_j G_{ik} G_{ik})
       - \partial_j G_{ik} \partial_j G_{ik} + \partial_j Q_{ijk} G_{ik}\\
&  = - \partial_k (\partial_j G_{ij} G_{ik}) + \partial_j (\partial_j G_{ik} G_{ik})
     - |\nabla G|^2 + \partial_j Q_{ijk} G_{ik}.
\end{align}
Therefore, we have
\begin{equation}
|\nabla G|^2 - |\nabla \cdot G|^2
\;=\; \partial_j (\partial_j G_{ik} G_{ik}) - \partial_k (\partial_j G_{ij} G_{ik}) + \partial_j Q_{ijk} G_{ik}.
\end{equation}
We next multiply by $\zeta^2$ and integrate
\begin{align}
\ltns{\zeta &\nabla G} - \ltns{\zeta \nabla \cdot G} \\
& = \int_{\rr^3} \zeta^2 \left[  \partial_j (\partial_j G_{ik} G_{ik}) - \partial_k (\partial_j G_{ij} G_{ik}) \right] dx
   \;+\; \int_{\rr^3} \zeta^2 \partial_j Q_{ijk} G_{ik} dx \\
& \lesssim   \int_{\rr^3} \zeta \jb{t}^{-1} |\nabla G| |G| dx
   \;+\;   \int_{\rr^3} \zeta \jb{t}^{-1} |Q_{ijk}| |G_{ik}| dx
      \;-\; \int_{\rr^3} \zeta^2 Q_{ijk} \partial_j  G_{ik} dx \\
& \leq \frac12 \ltns{\zeta \nabla G} + C \jb{t}^{-2} \ltns{ G} + C \ltns{\zeta Q},
\end{align}
where we have used Young's inequality and  $|\nabla \zeta^2| \lesssim \zeta \jb{t}^{-1} $. The statement
of the lemma follows immediately from this inequality.
\end{proof}

We now establish a higher order version of Lemma \ref{IntDecayLemma}.
Define
\begin{equation}
\label{LocEn2Def}
\ZZI{p,q}=\isum{p-1} \ltns{ \zeta \Delta S^k \Gamma^a v(t) },\quad q<p.
\end{equation}
\begin{prop}\label{IntLinEstHighOrd}
Assume that $\sigma$ in \eqref{CutOffs} is sufficiently small and that $\nu\le1$.
Fix $0\le q<p$.  Suppose that
\begin{equation}
S^k\Gamma^a H,\;  \nabla \cdot S^k\Gamma^a H,\; 
S^k \Gamma^aQ,\;  S^k \Gamma^a h  \in L^2([0,T];X^{p-k-1,0}),\; k=0,\ldots,q,
\end{equation}
for some $ 0<T<\infty$.  If $U=(G,v)$ is a solution of \eqref{Pde1Lin},\eqref{Pde2Lin},\eqref{Constr3Lin} such that
\begin{equation}
\sup_{0\le t\le T}\EE{p,q+1}<\infty,
\end{equation}
then for any  $0\le\theta\le1$,
\begin{multline}
\int_0^T\jb{t}^\theta\left[\YYI{p,q}+\nu^2\ZZI{p,q}\right]dt\\
\lesssim
\nu \jb{T}^{\theta-2}\EE[u][T]{p,q}
+\int_0^T\jb{t}^{\theta-2}\EE{p,q+1}dt\\+
\isum{p-1}\left\{
\int_0^T\jb{t}^\theta [
2 \nu  \ltip{\zeta  \nabla \cdot S^k \Gamma^a G}{\zeta  \nabla \cdot (S+1)^k \Gamma^aH} \right.   \\
\left.    + \ltns{\zeta (S+1)^k \Gamma^a Q} + \ltns{\zeta (S+1)^k \Gamma^aH} + \ltns{\zeta  (S+1)^k \Gamma^ah} ]dt
\vphantom{\int_0^T}\right\}.
\end{multline}
\end{prop}

\begin{proof}
The proof proceeds by induction on $q$.
Applying the vector fields $S^m \Gamma^a$, $m\le q$, $m+|a|\le p-1$, 
 to \eqref{Pde1Lin},\eqref{Pde2Lin},\eqref{Constr3Lin}
and using the commutation properties \eqref{LinCommut},\eqref{NonLinCommut},\eqref{ConstrVectrFld3},
we obtain the PDEs
%
\begin{align}
\label{TildePde1}
& \partial_t \td{G} - \nabla \td{v} = \td{H} \\
\label{TildePde2}
& \partial_t \td{v} - \nabla \cdot \td{G} - \nu \Delta \td{v} =  \td{h_0}\\
\label{TildeConstr3}
&\partial_k \td{G}_{ij} - \partial_j \td{G}_{ik}   =  \td{Q}_{ijk}.
\end{align}
%
where we have used the notation
\begin{equation} \label{TildeNotation}
\begin{aligned}
& \td{G} = S^m \Gamma^a G\\
&\td{v} = S^m \Gamma^a v \\
& \td{H} =(S+1)^m \Gamma^a H\\
& \td{Q}_{ijk} =(S+1)^m \Gamma^a Q_{ijk} \\
& \td{h}_0 = (S+1)^m \Gamma^a h - \sum_{j=0}^{m-1} (-1)^{m-j} {m \choose j} \nu B \Delta S^j \Gamma^a U.
\end{aligned}
\end{equation}

Fix $m=0$, and note that $\td{h}_0=\Gamma^a h$ (i.e. the sum is empty).  Apply Lemma \ref{IntDecayLemma} to 
\eqref{TildePde1}, \eqref{TildePde2},\eqref{TildeConstr3}
 and sum over $|a| \leq p-1$.  This directly yields the result when $q=0$.

Next, assume that the result holds for $q=q'-1$, with $q'<p$.  Take $m\le q'$, $m+|a|\le p-1$,
and note that
\[
\ltns{\zeta \td{h}_0}\lesssim \ltns{\zeta (S+1)^m \Gamma^a h}+\nu^2\ZZI{p,q'-1}.
\]
By the induction hypothesis, the quantity
\[
\nu^2\int_0^T\jb{t}^\theta\ZZI{p,q'-1} dt
\]
 has the desired bound,
so the result for $q=q'$ follows by summation over $m\le q'$, $m+|a|\le p-1$.

\end{proof}

\section{Local Energy Decay}
\label{DecayEstimates}
In this section we establish localized energy decay estimates for the nonlinear equation using a bootstrap argument
and an application of Proposition \ref{IntLinEstHighOrd}.
We remind the reader that the quantities $\YYI{p,q}$ and $\ZZI{p,q}$ were defined in \eqref{LocEnDef} and
\eqref{LocEn2Def}, respectively.

\begin{theorem}
\label{IntDecayNon}
Choose $(p,q)$ so that $p^\ast=\left[\frac{p+5}2\right]<q\le p$.
Suppose that $U=(G,v) \in C([0,T); X^{p,q})$ is a solution of \eqref{Pde1},\eqref{Pde2},\eqref{PressureFormula} with
\begin{equation}
\sup_{0\le t\le T}\EE{p,q}<\infty,
\end{equation}
and
\begin{equation}
\label{IntDecaySmallness}
\sup_{0\le t\le T} \EE{p^\ast,p^\ast}\le{\eps^2},
\end{equation}
for some $\eps$ sufficiently small.
Then
%
%
\begin{multline}
\label{IntDecayEst1}
\int_0^T\jb{t}^\theta\left[\YYI{p^\ast ,p^\ast -1}+\nu^2\ZZI{p^\ast ,p^\ast -1}\right]dt\\
\ \\
\lesssim
\begin{cases}
\sup\limits_{0\le t\le T} \jb{t}^{-\gamma}\EE{p^\ast ,p^\ast },&0<\theta+{\gamma}<1\\
\ \\
\log (e+T)\sup\limits_{0\le t\le T} \EE{p^\ast ,p^\ast },&\theta=1
\end{cases},
\end{multline}
and
\begin{multline}
\label{intdecayest2}
\int_0^T\jb{t}^\theta\left[\YYI{p^\ast +1,p^\ast }+\nu^2\ZZI{p^\ast +1,p^\ast }\right]dt\\
\lesssim \sup_{0\le t\le T}\jb{t}^{-{\gamma}}\EE{p,q},\quad 0<\theta+{\gamma}<1.
\end{multline}
%
%
\end{theorem}

\begin{proof}
Assume that the pair $(p,q)$ satisfies the hypotheses.  Then 
\[
p>p^\ast\ge5.
\] 
Choose any pair $(\pbar,\qbar)$ with
\begin{equation}
\label{bpbqh1}
\qbar<\pbar\le p,\quad \qbar<q.
\end{equation}
%
An application of Proposition \ref{IntLinEstHighOrd} yields
\begin{multline} \label{Ineq1}
\int_0^T\jb{t}^\theta\left[\YYI{\pbar ,\qbar}+\nu^2\ZZI{\pbar ,\qbar}\right]dt\\
\lesssim
\nu \jb{T}^{\theta-2}\EE[U][T]{\pbar ,\qbar}
+\int_0^T\jb{t}^{\theta-2}\EE{\pbar ,\qbar+1}dt+I+J,
\end{multline}
where
\begin{equation}
\label{Ineq1a}
I=\isum[\qbar]{\pbar -1}
\int_0^T \jb{t}^\theta 
2 \nu  \ltip{\zeta  \nabla \cdot S^k \Gamma^a G}{\zeta  \nabla \cdot (S+1)^k \Gamma^aH} dt,
\end{equation}
and
\begin{multline}
\label{Ineq1b}
J=\isum[\qbar]{\pbar -1}
\int_0^T \jb{t}^\theta 
\Big[ \ltns{\zeta (S+1)^k \Gamma^a Q} \\
+ \ltns{\zeta (S+1)^k \Gamma^aH} + \ltns{\zeta  (S+1)^k \Gamma^ah} \Big]dt,
\end{multline}
with the following inhomogeneous  terms:
\begin{align}
& H = \nabla v G - v \cdot \nabla G \\
\label{QuadDef}
& h = \nabla \cdot (GG^T) - v \cdot \nabla v - \nabla \pi \equiv h_1-\nabla\pi\\
& Q_{ij\ell} = G_{mj} \partial_m G_{i\ell} - G_{m\ell} \partial_m G_{ij}.
\end{align}
%

It follows from \eqref{NonLinCommut} and \eqref{ConstrVectrFld3} that
\begin{align}
\ltns{\zeta (S+1)^k \Gamma^aH} &+\ltns{\zeta (S+1)^k \Gamma^a h_1 } + \ltns{\zeta (S+1)^k \Gamma^a Q_{ij\ell}}\\
&\lesssim 
\sum_{\stacktwo{a_1+a_2=a}{k_1+k_2= k} }
\ltns{\zeta |S^{k_1}\Gamma^{a_1}U|\;|\nabla S^{k_2}\Gamma^{a_2}U|}\\
&\lesssim 
\sum_{\stacktwo{a_1+a_2=a}{k_1+k_2= k} }
\ltns{ |S^{k_1}\Gamma^{a_1}U|\;|\nabla S^{k_2}\Gamma^{a_2}U|},
\end{align}
since $|\zeta|\le1$.
For the pressure term, we have by \eqref{PressureCommut} and Lemma \ref{PressLemma},
\begin{align}
\ltns{\zeta (S+1)^k \Gamma^a\nabla\pi}
&\le\ltns{ (S+1)^k \Gamma^a\nabla\pi}\\
&=\ltns{\nabla S^k \Gamma^a\pi}\\
&=\ltns{\nabla\Delta^{-1}\partial_i\partial_jS^k \Gamma^a\left[(G G^T)_{ij} - v_iv_j \right]}\\
&\lesssim\sum_{\stacktwo{a_1+a_2=a}{k_1+k_2= k} }
\ltns{ |S^{k_1}\Gamma^{a_1}U|\;|\nabla S^{k_2}\Gamma^{a_2}U|},
\end{align}
since the operators $\Delta^{-1}\partial_i\partial_j$ are bounded on $L^2$.
From \eqref{Ineq1b}, this proves that 
\begin{equation}
\label{Jest}
J\lesssim \int_0^T \jb{t}^\theta R_{\pbar ,\qbar}^2dt,
\end{equation}
with
\begin{equation}
\label{Rdef}
R_{\pbar ,\qbar}^2\equiv \sum_{\stackthree{k_1+k_2+|a_1|+|a_2|\le\pbar }{k_1+k_2\le\qbar}{k_2+|a_2|<\pbar }}
 \ltns{ |S^{k_1}\Gamma^{a_1}U|\;|\nabla S^{k_2}\Gamma^{a_2}U|}.
\end{equation}

%
Next, we turn to the estimation of the integral $I$, defined in \eqref{Ineq1a}.
According to  \eqref{QuadDef} \and \eqref{NonLinCommut}, we have, adopting the notation  \eqref{TildeNotation}
\begin{multline}
\left|\nabla\cdot (S+1)^k \Gamma^a H -\nabla\cdot[\nabla\td{v}G-v\cdot\nabla\td{G}]\right|\\
\lesssim 
\sum_{\stackthree{a_1+a_2=a}{k_1+k_2= k} {k_2+|a_2|<k+|a|}}
{ \nabla\cdot[\nabla S^{k_2}\Gamma^{a_2}\td{v}\;S^{k_1}\Gamma^{a_1}G
-S^{k_1}\Gamma^{a_1}v\cdot\nabla \;S^{k_2}\Gamma^{a_2}\td{G}]}.
\end{multline}
It follows that
\begin{equation}
\label{Imain}
|I|\lesssim |I_1|+|I_2|+I_3,
\end{equation}
with
\begin{align}\label{IntDef1}
& I_1 =
\int_0^T\jb{t}^\theta  \nu  \ltip{\zeta  \nabla \cdot \td{G}}{\zeta  \nabla \cdot (\nabla \td{v} G) } dt  \\
& I_2 =
\int_0^T\jb{t}^\theta  \nu  \ltip{\zeta  \nabla \cdot \td{G}}{\zeta  \nabla \cdot (v \cdot \nabla \td{G}) } dt,
\end{align}
and
\begin{align}\label{IntDef2}
&I_3=
 \int_0^T  \jb{t}^\theta  \nu\; R_{\pbar ,\qbar}\;\YYIH{\pbar ,\qbar}\;dt.
 \end{align}

Let us first consider $I_1$.  Using the standard Sobolev inequality, we have
\[
\ltn[L^\infty]{U}+\ltn[L^\infty]{\nabla U}\lesssim\EEH{3,0}\le\EEH{p^\ast,p^\ast} < \eps,
\]
since $p^\ast\ge5>3$.  It follows that
\begin{align} \label{DETemp1}
 \nu  &|\ltip{\zeta  \nabla \cdot \td{G}}{\zeta  \nabla \cdot (\nabla \td{v} G) }|\\
&\lesssim \nu\ltn{\zeta\nabla\widetilde U}\Big(  \ltn[L^\infty]{\nabla U}\ltn{\zeta\nabla\td v}
+\ltn[L^\infty]{ U}\ltn{\zeta\nabla^2\td v} \Big)\\
&\lesssim\nu\eps\; \YYIH{\pbar ,\qbar}\Big( \YYIH{\pbar ,\qbar} + \ltn{\zeta\nabla^2\td v} \Big)\\
&\lesssim\eps \Big(\YYI{\pbar ,\qbar}+\nu^2\ltns{\zeta\nabla^2\td v}\Big).
\end{align}

For the final term in \eqref{DETemp1}, we derive a coercivity inequality.  Using integration
by parts, we have
\begin{align} \label{DETemp4}
 \ltns{\zeta \nabla^2 \td{v}}
&=  \sum_{k,m} \int_{\rr^3} \zeta^2  (\partial_k \partial_m \td{v}^i) ( \partial_k \partial_m \td{v}^i) dx \\
& = -\sum_{k,m}  \int_{\rr^3} (\partial_k \zeta^2)   (\partial_m \td{v}^i) (\partial_k \partial_m \td{v}^i) dx \\
& \hspace{1cm}   +   \sum_{k,m} \int_{\rr^3}  (\partial_m \zeta^2)   (\partial_m \td{v}^i) (\Delta   \td{v}^i) dx \\
& \hspace{1cm}   +   \sum_{k,m} \int_{\rr^3}   \zeta^2   |\Delta \td{v}|^2 dx.
\end{align}
Using \eqref{WeightDer}, we obtain $|\nabla \zeta^2| \lesssim \zeta \jb{t}^{-1}$, and so \eqref{DETemp4} yields
\begin{align} \label{DETemp5}
 \ltns{\zeta \nabla^2 \td{v}}
& \lesssim \jb{t}^{-1}\ltn{\nabla \td v}\ltn{\zeta\nabla^2\td v}+\ltns{\zeta\Delta\td v}\\
&\lesssim M^{-1}\ltns{\zeta\nabla^2\td v}+M\jb{t}^{-2}\ltns{\nabla \td v}+\ltns{\zeta\Delta\td v},
\end{align}
for any $M>0$.  Choosing $M$ sufficiently large, we obtain the bound
\begin{align}
\label{DETemp6}
\ltns{\zeta \nabla^2 \td{v}}&\lesssim \jb{t}^{-2}\ltns{\nabla \td v}+\ltns{\zeta\Delta\td v}\\
&\le \jb{t}^{-2} \EE{\pbar ,\qbar} + \ZZI{\pbar ,\qbar}.
\end{align}

Altogether, from \eqref{DETemp1} and \eqref{DETemp6}, we conclude that
\begin{equation}
\label{I1est}
|I_1|\lesssim\eps\Big(\jb{t}^{-2}\EE{\pbar ,\qbar}+\YYI{\pbar ,\qbar}+\nu^2\ZZI{\pbar ,\qbar}\Big).
\end{equation}

To estimate  $I_2$, we write
\begin{align}\label{DETemp7}
&  |\ltip{\zeta  \nabla \cdot \td{G}}{\zeta  \nabla \cdot (v \cdot \nabla \td{G}) }|\\
&\qquad \;=\;  \left|\int_{\rr^3} \zeta^2 (\nabla \cdot \td{G})_i \partial_k (v_m \partial_m \td{G}_{ik}) dx\right|\\
&\qquad \;\le\;  \left|\int_{\rr^3} \zeta^2 (\nabla \cdot \td{G})_i(\partial_k v_m) ( \partial_m \td{G}_{ik})dx\right|\\
&\qquad\qquad +\left|\int_{\rr^3} \zeta^2(\nabla \cdot \td{G})_i(v\cdot \nabla) (\nabla \cdot \td{G})_i dx\right|.
\end{align}
The first term is easily estimated by
\begin{align}
\left|\int_{\rr^3} \zeta^2 (\nabla \cdot \td{G})_i(\partial_k v_m) ( \partial_m \td{G}_{ik})dx\right|
&\le \ltn[L^\infty]{\nabla U}\ltns{\zeta\nabla\td U}\\
&\lesssim\eps\YYI{\pbar ,\qbar}.
\end{align}
For the second term, we use integration by parts and the fact that $\nabla\cdot v=0$
\begin{align}
&\left|\int_{\rr^3} \zeta^2(\nabla \cdot \td{G})_i(v\cdot \nabla) (\nabla \cdot \td{G})_i dx\right|\\
&\qquad=\left|\tfrac12\int_{\rr^3} \zeta^2(v\cdot \nabla)|\nabla \cdot \td{G}|^2dx\right|\\
&\qquad=\left|-\tfrac12\int_{\rr^3} [(v\cdot \nabla)\zeta^2]|\nabla \cdot \td{G}|^2dx\right|\\
&\qquad\lesssim\jb{t}^{-1}\int_{\rr^3} \zeta |U| |\nabla \td G|^2dx\\
&\qquad \lesssim \jb{t}^{-1} \ltn[L^\infty]{U}\ltn{\nabla\td G}\ltn{\zeta\nabla\td G}\\
&\qquad \lesssim \eps \jb{t}^{-1} \EEH{\pbar ,\qbar}\YYIH{\pbar ,\qbar}\\
&\qquad \lesssim \eps\Big(\jb{t}^{-2}\EE{\pbar ,\qbar}+\YYI{\pbar ,\qbar}\Big).
\end{align}
This shows that 
\begin{equation}
\label{I2est}
I_2\lesssim\eps\Big(\jb{t}^{-2}\EE{\pbar ,\qbar}+\YYI{\pbar ,\qbar}+\nu^2\ZZI{\pbar ,\qbar}\Big).
\end{equation}

Since $\nu\le1$, we have by Young's inequality, that
\begin{equation}
\label{I3est}
I_3\le M^{-1}\int_0^T\jb{t}^\theta\YYI{\pbar ,\qbar}\;dt+M\int_0^T  \jb{t}^\theta  R_{\pbar ,\qbar}^2\;dt.
\end{equation}

Combining \eqref{Imain}, \eqref{I3est}, \eqref{I1est}, and \eqref{I2est}, we conclude  that
\begin{multline}
\label{Iest}
I\lesssim \int_0^T \jb{t}^{\theta-2}\EE{\pbar ,\qbar}\\
+(\eps +M^{-1})\int_0^T\jb{t}^{\theta}\Big(\YYI{\pbar ,\qbar}+\nu^2\ZZI{\pbar ,\qbar}\Big) dt\\
+M\int_0^T\jb{t}^{\theta} R_{\pbar ,\qbar}^2dt.
\end{multline}

Upon assembling the estimates  \eqref{IntDef2},\eqref{Ineq1},\eqref{Jest},\eqref{Iest},  we obtain
\begin{multline} \label{Ineq2}
\int_0^T\jb{t}^\theta\left[\YYI{\pbar ,\qbar}+\nu^2\ZZI{\pbar ,\qbar}\right]dt\\
\lesssim
\nu \jb{T}^{\theta-2}\EE[U][T]{\pbar ,\qbar}
+\int_0^T\jb{t}^{\theta-2}\EE{\pbar ,\qbar+1}dt\\
+(\eps +M^{-1})\int_0^T\jb{t}^{\theta}\Big(\YYI{\pbar ,\qbar}+\nu^2\ZZI{\pbar ,\qbar}\Big) dt\\
+M\int_0^T\jb{t}^{\theta} R_{\pbar ,\qbar}^2dt.
\end{multline}
By taking $\eps$ sufficiently small and $M$ sufficiently large, the middle integral can be absorbed on the left,
resulting in
\begin{multline} 
\label{Ineq3}
\int_0^T\jb{t}^\theta\left[\YYI{\pbar ,\qbar}+\nu^2\ZZI{\pbar ,\qbar}\right]dt\\
\lesssim
\nu \jb{T}^{\theta-2}\EE[U][T]{\pbar ,\qbar}
+\int_0^T\jb{t}^{\theta-2}\EE{\pbar ,\qbar+1}dt
+\int_0^T\jb{t}^{\theta} R_{\pbar ,\qbar}^2dt.
\end{multline}
Recalling the definition \eqref{Rdef}
 and applying  Lemmas \ref{IntCalcIneq},\ref{ExtCalcIneq} we have that
\begin{equation}
\label{Ineq4}
R_{\pbar ,\qbar}^2\lesssim \Bigl(\YYI{\pbar ',\qbar'}+\jb{t}^{-2}\EE{\pbar ',\qbar'}\Bigr)\EE{\pbar ,\qbar},
\end{equation}
where
\[
\pbar '= \left[\frac{\pbar +5}2\right],\quad \qbar'=\left[\frac{\pbar }2\right].
\]
The estimates \eqref{Ineq3}, \eqref{Ineq4} imply that
\begin{multline} 
\label{Ineq5}
\int_0^T\jb{t}^\theta\left[\YYI{\pbar ,\qbar}+\nu^2\ZZI{\pbar ,\qbar}\right]dt\\
\lesssim
\nu \jb{T}^{\theta-2}\EE[U][T]{\pbar ,\qbar}
+\int_0^T\jb{t}^{\theta-2}\EE{\pbar ,\qbar+1}dt\\
+\int_0^T\jb{t}^{\theta} \Bigl(\YYI{\pbar ',\qbar'}+\jb{t}^{-2}\EE{\pbar ',\qbar'}\Bigr)\EE{\pbar ,\qbar}dt.
\end{multline}

Now choose $(\pbar ,\qbar)=(p^\ast,p^\ast-1)$, which is consistent with the requirement
\eqref{bpbqh1}.  Since $p^\ast\ge5$, there holds
\[
 \pbar '=\left[\frac{p^\ast+5}2\right] \le p^\ast, \quad \qbar'=\left[\frac{p^\ast}2\right]\le p^\ast-1.
 \]
In this case, our estimate \eqref{Ineq5} yields
\begin{multline} 
\label{Ineq6}
\int_0^T\jb{t}^\theta\left[\YYI{p^\ast,p^\ast-1}+\nu^2\ZZI{p^\ast,p^\ast-1}\right]dt\\
\lesssim
\nu \jb{T}^{\theta-2}\EE[U][T]{p^\ast,p^\ast}
+\int_0^T\jb{t}^{\theta-2}\EE{p^\ast,p^\ast}dt\\
+\int_0^T\jb{t}^{\theta} \Bigl(\YYI{p^\ast,p^\ast-1}+\jb{t}^{-2}\EE{p^\ast,p^\ast}\Bigr)\EE{p^\ast,p^\ast}dt.
\end{multline}
By the assumption \eqref{IntDecaySmallness}, we obtain 
\begin{multline} 
\label{Ineq7}
\int_0^T\jb{t}^\theta\left[\YYI{p^\ast,p^\ast-1}+\nu^2\ZZI{p^\ast,p^\ast-1}\right]dt\\
\lesssim
\nu \jb{T}^{\theta-2}\EE[U][T]{p^\ast,p^\ast}
+\int_0^T\jb{t}^{\theta-2}\EE{p^\ast,p^\ast}dt\\
+\eps\int_0^T\jb{t}^{\theta} \YYI{p^\ast,p^\ast-1}dt.
\end{multline}
The last integral can be absorbed on the left, if $\eps$ is small enough, and we obtain finally
\begin{multline} 
\label{Ineq8}
\int_0^T\jb{t}^\theta\left[\YYI{p^\ast,p^\ast-1}+\nu^2\ZZI{p^\ast,p^\ast-1}\right]dt\\
\lesssim
\nu \jb{T}^{\theta-2}\EE[U][T]{p^\ast,p^\ast}
+\int_0^T\jb{t}^{\theta-2}\EE{p^\ast,p^\ast}dt.
\end{multline}
 Choose $\gamma\ge0$ so that $0<\theta+\gamma\le1$.
From \eqref{Ineq8} we obtain
\begin{multline} 
\label{Ineq9}
\int_0^T\jb{t}^\theta\left[\YYI{p^\ast,p^\ast-1}+\nu^2\ZZI{p^\ast,p^\ast-1}\right]dt\\
\lesssim \sup_{0\le t\le T}\jb{t}^{-\gamma}\EE{p^\ast,p^\ast}
\left[\jb{T}^{\theta+\gamma-2}+\int_0^T\jb{t}^{\theta+\gamma-2}dt\right],
\end{multline}
from which the results \eqref{IntDecayEst1} follow directly.

Finally, going back to \eqref{Ineq5},  we 
take $(\pbar ,\qbar)=(p^\ast+1,p^\ast)$, which is also consistent with \eqref{bpbqh1}.
Note that
\[
\pbar =p^\ast+1\le p,\quad \qbar+1=p^\ast+1\le p,
\]
and
\[
\pbar '=\left[\frac{p^\ast+6}2\right]\le p^\ast,\quad \qbar'=\left[\frac{p^\ast+1}2\right]\le p^\ast-1,
\]
since $p^\ast\ge5$.  
We get from \eqref{Ineq5}
\begin{align} 
\label{Ineq10}
\int_0^T\jb{t}^\theta&\left[\YYI{p^\ast+1,p^\ast}+\nu^2\ZZI{p^\ast+1,p^\ast}\right]dt\\
\lesssim\;&
\nu \jb{T}^{\theta-2}\EE[U][T]{p,q}
+\int_0^T\jb{t}^{\theta-2}\EE{p,q}dt\\
&+\int_0^T\jb{t}^{\theta} \Bigl(\YYI{p^\ast,p^\ast-1}+\jb{t}^{-2}\EE{p^\ast,p^\ast}\Bigr)\EE{p,q}dt.
\end{align}
Let $\gamma\ge0$ with $0<\theta+\gamma<1$.  By  \eqref{IntDecayEst1} and \eqref{IntDecaySmallness}, 
this yields
\begin{align} 
\label{Ineq11}
\int_0^T\jb{t}^\theta&\left[\YYI{p^\ast+1,p^\ast}+\nu^2\ZZI{p^\ast+1,p^\ast}\right]dt\\
\lesssim\;&
\sup_{0\le t\le T}\jb{t}^{-\gamma}\EE{p,q}
\Bigg[ \jb{T}^{\theta+\gamma-2}
+\int_0^T\jb{t}^{\theta+\gamma-2}dt\\
&+\int_0^T\jb{t}^{\theta+\gamma}\YYI{p^\ast,p^\ast-1}dt\\
&+\sup_{0\le t\le T}\EE{p^\ast,p^\ast}\int_0^T\jb{t}^{\theta+\gamma-2}dt\Bigg]\\
\lesssim\;&\sup_{0\le t\le T}\jb{t}^{-\gamma}\EE{p,q}\;\left[1+\sup_{0\le t\le T}\EE{p^\ast,p^\ast}\right]\\
\lesssim\;&\sup_{0\le t\le T}\jb{t}^{-\gamma}\EE{p,q}\;[1+\eps^2]\\
\lesssim\;&\sup_{0\le t\le T}\jb{t}^{-\gamma}\EE{p,q}.
\end{align}
This completes the proof of \eqref{intdecayest2}.

\end{proof}

\section{Basic Energy Estimate}

\begin{lemma}
Suppose that $U=(G,v) \in C([0,T); X^{p,q})$ is a solution of \eqref{Pde1},\eqref{Pde2},\eqref{PressureFormula}.
Then
\begin{multline}\label{BEest}
\EE[U][T]{p,q} \;\lesssim\; \EEZ{p,q}\\
     \sum_{\stackthree{\stacktwo{a_1+a_2=a}{k_1+k_2= k}}{|a|+k\le p}{k\le q}}
\left|
   \int_0^T
      \ltip  {   N(  S^{k_1}\Gamma^{a_1} U(t), \nabla S^{k_2}\Gamma^{a_2} U(t)  )   }   {S^k \Gamma^a U(t)}dt
          \right|.
\end{multline}
\end{lemma}

\begin{proof}
Taking the $L^2$ inner product of
\begin{equation}
LU = \partial_t U - A(\nabla) U - \nu B \Delta U
\end{equation}
with $U = (G,v)$, we obtain
\begin{multline}\label{HEtemp1}
\ltip{\partial_t U(t)}{U(t)} - \ltip{A(\nabla) U(t)}{U(t)} \\
- \ltip{\nu B \Delta U(t)}{U(t)}
= \ltip{LU(t)}{U(t)}.
\end{multline}
The second term on the left of \eqref{HEtemp1} vanishes
\begin{multline}
\ltip{A(\nabla) U(t)}{U(t)} = - \ltip{\nabla v(t)}{G(t)} - \ltip{\nabla \cdot G(t)}{v(t)} \\
= - \ltip{\nabla v(t)}{G(t)} + \ltip{G(t)}{\nabla v(t)} = 0.
\end{multline}
Using integration by parts, the third term on left of \eqref{HEtemp1} can be written as
\begin{multline}
\ltip{\nu B \Delta U(t)}{U(t)} \;=\; \ltip{\nu \Delta v(t)}{v(t)}
 \;=\; \int_{\rr^3} \nu \partial_k^2 v_i(t) v_i(t) dx \\
 = - \int_{\rr^3} \nu \partial_k v_i(t) \partial_k v_i(t) dx
 \;=\; - \nu \ltns{\nabla v(t)}.
\end{multline}
Therefore, \eqref{HEtemp1} becomes
\begin{equation}
 \frac12 \partial_t \ltns{U(t)} + \nu \ltns{\nabla v(t)} = \ltip{LU(t)}{U(t)}.
\end{equation}
Integration over time gives
\begin{equation}
\frac12 \ltns{U(T)} + \nu \int_0^T \ltns{\nabla v(t)} dt
    \;=\; \frac12 \ltns{U(0)} +  \int_0^T \ltip{LU(t)}{U(t)}dt,
\end{equation}
which implies that for $0 \leq T < T_0$
\begin{align}
&  \EE[U][T]{0,0} = \EEZ{0,0} + \int_0^T \ltip{LU(t)}{U(t)}dt.
\end{align}
For $p \geq q \geq 0$, we apply the above estimate to higher order derivatives $S^k\Gamma^aU$
and together
with the commutation property \eqref{LinCommut} we obtain
\begin{multline}\label{HEtemp4}
\EE[U][T]{p,q} = \EEZ{p,q} + I\\
    + \isum{p} \int_0^T \ltip{(S+1)^k \Gamma^a LU(t)}{S^k \Gamma^a U(t)}dt,
\end{multline}
where
\begin{align}
& I = - \isum{p} \sum_{j=0}^{k-1} \int_0^T
\left \langle   (-1)^{k-j} {k \choose j} \nu B \Delta S^j \Gamma^a U(t)  ,  S^k \Gamma^a U(t)   \right \rangle_{L_2}  dt.
\end{align}
For $q>0$ we show by integration by parts that
\begin{align}
\int_0^T \ltip{\nu B \Delta & S^j \Gamma^a U(t)}{S^k \Gamma^a U(t)}dt \\
& = \int_0^T \int_{\rr^3} \nu \Delta (S^j \Gamma^a v(t))^i (S^k \Gamma^a v(t))^i dx dt \\
& = \; - \int_0^T \int_{\rr^3} \nu \nabla (S^j \Gamma^a v(t))^i \cdot \nabla (S^k \Gamma^a v(t))^i dx dt \\
& \leq  \int_0^T \nu  \ltn{\nabla S^j \Gamma^a v(t)} \ltn{\nabla S^k \Gamma^a v(t)} dt \\
& \leq     \left(  \nu \int_0^T  \ltns{\nabla S^j \Gamma^a v(t)} dt \right)^{1/2}
                          \left(  \nu \int_0^T  \ltns{\nabla S^k \Gamma^a v(t)} dt \right)^{1/2}.
\end{align}
Therefore, we have
\begin{equation}
I \lesssim \EEH[U][T]{p,q-1}\EEH[U][T]{p,q}.
\end{equation}
Applying Young's inequality to the above bound and substituting into \eqref{HEtemp4} give
\begin{multline}
\EE[U][T]{p,q} \; \lesssim \; \EEZ{p,q} + \EE[U][T]{p,q-1} + \lambda \EE[U][T]{p,q}  \\
    + \isum{p} \int_0^T \ltip{(S+1)^k \Gamma^a LU(t)}{S^k \Gamma^a U(t)}dt,
\end{multline}
where $\lambda$ is a small enough constant so that the corresponding energy term can be absorbed on the left.
Induction on $q$ further gives
\begin{multline}\label{HEtemp4.5}
\EE[U][T]{p,q} \;\lesssim\; \EEZ{p,q}\\
     + \isum{p} \int_0^T \ltip{(S+1)^k \Gamma^a LU(t)}{S^k \Gamma^a U(t)}dt.
\end{multline}
We recall that by Lemma \ref{CommutLemma}, \eqref{PressureCommut}
\begin{equation}
(S+1)^k \Gamma^a LU = (S+1)^k \Gamma^a N(U,\nabla U) + (0, -\nabla S^k \Gamma^a  \pi ).
\end{equation}
Using integration by parts and the constraint \eqref{ConstrVectrFld1}, we see
that the pressure term vanishes
\begin{equation}
 -\ltip{ \nabla S^k \Gamma^a \pi}{S^k \Gamma^a v} = \ltip{S^k \Gamma^a \pi}{\nabla\cdot S^k \Gamma^a v}=0.
\end{equation}

Therefore, by \eqref{NonLinCommut} we can write the energy inequality \eqref{HEtemp4.5} as
\begin{multline} 
\EE[U][T]{p,q} \lesssim \EEZ{p,q}  \\
+
\sum_{\stackthree{\stacktwo{a_1+a_2=a}{k_1+k_2= k}}{|a|+k\le p}{k\le q}}
\left|
   \int_0^T
      \ltip  {   N(  S^{k_1}\Gamma^{a_1} U(t), \nabla S^{k_2}\Gamma^{a_2} U(t)  )   }   {S^k \Gamma^a U(t)}dt
          \right|.
\end{multline}

\end{proof}

\section{High Energy Estimates} \label{HighEnrgSect}

In this section we estimate the top order energy.  Here it is crucial to avoid loss of derivatives.

\begin{prop} \label{HighEnrgEst}
Choose $(p,q)$ so that $5\le p^\ast=\left[\frac{p+5}2\right]\le q\le p$.
Suppose that $U=(G,v)\in C([0,T_0),X^{p,q})$ is a solution of \eqref{Pde1},\eqref{Pde2},\eqref{PressureFormula} with
\begin{equation}
\label{smallness}
\sup_{0\le t< T_0}\EE{p^\ast,p^\ast}\le{\eps^2},
\end{equation}
for some $\eps$ sufficiently small.
Then there exists a constant $C_1>1$ such that
\begin{align}
&\EE{p,q}\le C_1\EEZ{p,q}\jb{t}^{C_1{\eps}}\\
&\EE{p^\ast,p^\ast}\le C_1 \EEZ{p^\ast,p^\ast}\jb{t}^{C_1{\eps}},
\end{align}
for $0\le t<T_0$.
\end{prop}

\begin{proof}

Using \eqref{PdeMatrx}, \eqref{DefN}, and \eqref{DefN1N2}, we can write the inner products in \eqref{BEest} as
\begin{align}
&\ltip{N(  S^{k_1}\Gamma^{a_1} U(t), \nabla S^{k_2}\Gamma^{a_2} U(t)  )   }   {S^k \Gamma^a U(t)}\\
& \qquad=\ltip{ \nabla S^{k_2}\Gamma^{a_2}v S^{k_1}\Gamma^{a_1}G}{S^k \Gamma^a G(t)}\\
&\qquad\quad -\ltip{S^{k_1}\Gamma^{a_1}v\cdot\nabla S^{k_2}\Gamma^{a_2}G  }   {S^k \Gamma^a G(t)}\\
&\qquad\quad +\ltip{\nabla\cdot (S^{k_1}\Gamma^{a_1}G S^{k_2}\Gamma^{a_2} G^\top)}{S^k \Gamma^a v(t)}\\
&\qquad\quad -\ltip {S^{k_1} \Gamma^{a_1}v\cdot\nabla S^{k_2}\Gamma^{a_2}v}{S^k \Gamma^a v(t)}.
\end{align}

We will first address the special case when $a_1=0$ and $k_1=0$, i.e.\ $a_2=a$ and $k_2=k$.
Using the notation  $\td{v}=S^k \Gamma^a v$
and $\td{G}=S^k \Gamma^a G$ introduced in \eqref{TildeNotation}, we start with
\begin{equation}
-\ltip{v \cdot \nabla \td{G}}{\td{G}}
     \;=\; - \int_{\rr^3} v _k \partial_k \td{G}_{ij} \td{G}_{ij} dx
  \;=\; - \frac12 \int_{\rr^3} v _k \partial_k |\td{G}|^2 dx \;=\; 0,
\end{equation}
by \eqref{ConstrVectrFld1}.
Similarly, we see that $\ltip{-v \cdot \nabla \td{v}}{\td{v}} = 0$.

The remaining two terms are
\begin{multline}
\ltip{\nabla\td{v}G}{\td{G}}+\ltip{\nabla\cdot(\td{G}G^\top)}{\td{v}}\\
=\int_{\rr^3}[\partial_j\td{v}_iG_{jk}\td{G}_{ik}+\partial_j(\td{G}_{ik}G_{jk})\td{v}_i]dx
\int_{\rr^3}\partial_j(\td{v}_i\td{G}_{ik}G_{jk})dx=0.
\end{multline}

This shows that all of the the terms in the sum with $a_1=0$ and $k_1=0$ cancel.
Therefore we can write \eqref{BEest} as
\begin{multline} \label{LossDrvtv}
\EE[U][T]{p,q} \lesssim \EEZ{p,q}  \\
 + \sum_{\stackthree{|a_1+a_2|+k_1+k_2\le p}{k_1+k_2\le q}{|a_2|+k_2<p}}
\int_0^T
\ltn{\; |S^{k_1} \Gamma^{a_1} U(t)| \; |S^{k_2} \Gamma^{a_2+1} U(t)| \;}
\EEH{p,q} dt,
\end{multline}
and thus, there is no derivative loss.

 By \eqref{Unity} and Lemmas
\ref{IntCalcIneq} and \ref{ExtCalcIneq}, we get the  bound
\begin{align}
&\hspace*{-15mm}\sum_{\stackthree{|a_1+a_2|+k_1+k_2\le p}{k_1+k_2\le q}{|a_2|+k_2<p}}
\ltn{\;|S^{k_1}\Gamma^{a_1}u(t)|\; | S^{k_2}\Gamma^{a_2+1}u(t)|\;}\\
\lesssim
&\sum_{\stackthree{|a_1+a_2|+k_1+k_2\le p}{k_1+k_2\le q}{|a_2|+k_2<p}}
\ltn{\zeta \;|S^{k_1}\Gamma^{a_1}u(t)|\; | S^{k_2}\Gamma^{a_2+1}u(t)|\;}\\
&+\sum_{\stackthree{|a_1+a_2|+k_1+k_2\le p}{k_1+k_2\le q}{|a_2|+k_2<p}}
\ltn{\eta \;|S^{k_1}\Gamma^{a_1}u(t)|\; | S^{k_2}\Gamma^{a_2+1}u(t)|\;}\\
\lesssim
&\left[\YYIH{\left[\frac{p+5}2\right],\left[\frac{p+3}2\right]}+\jb{t}^{-1}\EEH{\left[\frac{p+5}2\right],\left[\frac{p+5}2\right]}\right]\EEH{p,q}.
\end{align}
Inserting this into \eqref{LossDrvtv} yields
\begin{multline}
\label{energyid5}
\EE[u][T]{p,q}\lesssim \EEZ{p,q}\\+
\int_0^T\left[\YYIH{\left[\frac{p+5}2\right],\left[\frac{p+3}2\right]}+\jb{t}^{-1}\EEH{\left[\frac{p+5}2\right],\left[\frac{p+5}2\right]}\right]
\EE{p,q}dt.
\end{multline}
An application of Gronwall's inequality produces
\begin{multline}
\label{energyid6}
\EE[u][T]{p,q}\\
\lesssim \EEZ{p,q}
\exp
\int_0^T\left[\YYIH{\left[\frac{p+5}2\right],\left[\frac{p+3}2\right]}\!+\!\jb{t}^{-1}\EEH{\left[\frac{p+5}2\right],\left[\frac{p+5}2\right]}\right]dt.
\end{multline}
We point out that \eqref{energyid6} holds for any pair $(p,q)$ as long as $p\ge q\ge \left[\frac{p+5}2\right]\ge5$,
which requires only $p\ge 5$.

Recalling the definition $p^\ast=\left[\frac{p+5}2\right]$, we obtain using Theorem \ref{IntDecayNon}
\begin{align}
\int_0^T 
\YYIH{\left[\frac{p+5}2\right],\left[\frac{p+3}2\right]}dt 
&\le
\left(\int_0^T\jb{t}\YYI {p^\ast,p^\ast-1}dt\right)^{1/2}\left(\int_0^T\jb{t}^{-1}dt\right)^{1/2}\\
&\lesssim\left( \sup_{0\le t\le T}\EE{p^\ast,p^\ast}\log(e+T)\right)^{1/2}\left(\log(e+T)\right)^{1/2}\\
&\lesssim \sup_{0\le t\le T}\EEH{p^\ast,p^\ast}\log(e+T).
\end{align}
Likewise, we have the bound
\begin{equation}
\int_0^T\jb{t}^{-1}\EEH{\left[\frac{p+5}2\right],\left[\frac{p+5}2\right]}
\lesssim \sup_{0\le t\le T}\EEH{p^\ast,p^\ast}\log(e+T).
\end{equation}
Now thanks to the assumption \eqref{smallness}, the inequality \eqref{energyid6} implies that
\begin{equation}
\EE[u][T]{p,q}
\lesssim \EEZ{p,q}
\exp \left[C{\eps}\log(e+T)\right]\le  \EEZ{p,q}\jb{T}^{C_1{\eps}}.
\end{equation}

Returning to \eqref{energyid6}, we can repeat this argument with the pair $(p,q)=(p^\ast,p^\ast)$
because $p^\ast\ge5$ implies that $ p^\ast\ge\left[\frac{p^\ast+5}2\right]$.  Therefore, we also obtain the
bound
\begin{equation}
\EE[u][T]{p^\ast,p^\ast}
\lesssim
  \EEZ{p^\ast,p^\ast}\jb{T}^{C_1{\eps}},
\end{equation}
after a possible increase in the size of the constant $C_1$.
Note however, that the choice of $C_1$ is independent of ${\eps}$.
Theorem \ref{HighEnrgEst} now follows.

\end{proof}

\section{Low Energy Estimates}

Now we will estimate the lower energy where the focus will be to obtain the best possible
temporal decay. 

\begin{prop}
Choose  $(p,q)$ such that $p\ge11$,  and  $p\ge q> p^\ast$, where $p^\ast=\left[\frac{p+5}2\right]$.
Suppose that $U=(G,v)\in C([0,T_0),X^{p,q})$ is a solution of \eqref{Pde1},\eqref{Pde2},\eqref{PressureFormula} with
\begin{equation}
\label{SmallnessEps}
\sup_{0\le t< T_0}\EE{p^\ast,p^\ast}\le{\eps^2}\ll1.
\end{equation}
There exists a constant $C_0>1$ such that
\begin{equation}
\label{Conclusion}
\sup_{0\le t<T_0}\EE{p^\ast,p^\ast}\\
\le C_0\EEZ{p^\ast,p^\ast}
\left(1+\EEZH{p,q}\right).
\end{equation}
\end{prop}

\begin{proof}
We start with \eqref{BEest} applied to $(p,q) = (p^\ast,p^\ast)$

\begin{multline} \label{GlbEnrgInq}
\EE[U][T]{p^\ast,p^\ast} \lesssim \EEZ{p^\ast,p^\ast}  \\
+
\sum_{\stackthree{\stacktwo{a_1+a_2=a}{k_1+k_2= k}}{|a|+k\le p^\ast}{k\le p^\ast}}
\left|
   \int_0^T
      \ltip  {   N(  S^{k_1}\Gamma^{a_1} U(t), \nabla S^{k_2}\Gamma^{a_2} U(t)  )   }   {S^k \Gamma^a U(t)}dt
          \right|.
\end{multline}
We have shown in Section \eqref{HighEnrgSect} that the summation indices satisfy $|a_2| + k_2 < p^\ast$, but we will not need
this fact here.

Using the cut-off functions \eqref{CutOffs} and the property \eqref{Unity}, we can bound
the integral on the right-hand side by
\begin{align} 
&\sum_{\stackthree{\stacktwo{a_1+a_2=a}{k_1+k_2= k}}{|a|+k\le p^\ast}{k\le p^\ast}}
\left|
 \int_0^T
    \ltip  {   N(  S^{k_1} \Gamma^{a_1} U(t), \nabla S^{k_2}\Gamma^{a_2} U(t)  )   }   {S^k \Gamma^a U(t)}dt
\right| \\
& \hs{1} \lesssim
\sum_{\stackthree{\stacktwo{a_1+a_2=a}{k_1+k_2= k}}{|a|+k\le p^\ast}{k\le p^\ast}}
\int_0^T \int_{\rr^3}
    \zeta^2 |S^{k_1}\Gamma^{a_1} U(t)|\;|\nabla S^{k_2}\Gamma^{a_2} U(t)|\;|S^k \Gamma^a U(t)|dxdt \\
    \label{GlbI1I2Inq}
&  \hs{2} +
\sum_{\stackthree{\stacktwo{a_1+a_2=a}{k_1+k_2= k}}{|a|+k\le p^\ast}{k\le p^\ast}}
\left|
 \int_0^T
    \ltip  {  \eta N(  S^{k_1}\Gamma^{a_1} U(t), \nabla S^{k_2}\Gamma^{a_2} U(t)  )   }   {S^k \Gamma^a U(t)}dt
\right|\\
&  \hs{1} \equiv I_1 + I_2,
\end{align}
where $I_1$ and $I_2$ denote correspondingly the two integrals on the right and $T$ is in the range
$0 \leq T < T_0$.

We shall show that
\begin{equation} \label{GlbI1}
I_j \lesssim \EEZ{p^\ast,p^\ast} \EEZH{p,q}\quad j=1,2.
\end{equation}

The interior intergral $I_1$ can be bounded by
\begin{multline}
I_1 =
\sum_{\stackthree{\stacktwo{|a_1+a_2|\le |a|}{k_1+k_2= k}}{|a|+k\le p^\ast}{k\le p^\ast}}
\int_0^T
\int_{\rr^3}  \zeta ^ 2
|   S^{k_1} \Gamma^{a_1} U|\;
|\nabla   S^{k_2} \Gamma^{a_2}U|\;
| S^k \Gamma^a U | \; dxdt \\
\lesssim
\sum_{\stacktwo{k_1+k_2+|a_1|+|a_2|\le p^\ast}{k_1+k_2\le p^\ast}}
\int_0^T
\ltn{
\; \zeta ^ 2
|S^{k_1} \Gamma^{a_1} U|\;
| \nabla S^{k_2} \Gamma^{a_2}U|\;
}\;
\EEH{p^\ast,p^\ast}  dt.
\end{multline}

In the case $k_1+|a_1| <  k_2+|a_2|+1$, i.e.\ $k_1+|a_1| \leq
 \left[   \frac{p^\ast }{2}  \right] $, using  \eqref{SobInt1}, we have
\begin{align}
\ltn{
\zeta ^ 2
|S^{k_1} &\Gamma^{a_1} U|
| \nabla S^{k_2} \Gamma^{a_2}U|
\;}
\\
&\lesssim
\ltn[L^\infty]{\zeta S^{k_1} \Gamma^{a_1} U}
\ltn{ \zeta  \nabla S^{k_2} \Gamma^{a_2}U}
\\
&\lesssim
\left [  \YYIH{  \left[  \frac{p^\ast +4 }{2}  \right] ,  \left[ \frac{p^\ast }{2}\right]   }
+
\jb{t}^{-1}    \EEH{  \left[  \frac{p^\ast +2 }{2}  \right] ,  \left[ \frac{p^\ast }{2}\right]   }      \right]
    \YYIH{p^\ast+1,p^\ast}.
\end{align}

We next consider the case when $k_2+|a_2|+1 \leq k_1+|a_1|$, i.e.
$k_2+|a_2| \leq \left[   \frac{p^\ast - 1}{2}  \right]$.
With the use of Hardy's inequality \eqref{SobHrdy} and the Sobolev inequality \eqref{SobInt2}, we have
\begin{align}
\ltn{ \zeta ^ 2
  \; |S^{k_1} \Gamma^{a_1} U|\;
&| \nabla S^{k_2} \Gamma^{a_2}U|\;}\\
&\lesssim  \ltn{r^{-1}\zeta S^{k_1} \Gamma^{a_1} U}
\ltn[L^\infty]
{r\zeta \nabla S^{k_2} \Gamma^{a_2}U}\\
&\lesssim
\left[\YYIH{p^{\ast}+1,p^\ast}+\jb{t}^{-1}\EEH{p^\ast,p^\ast}\right]
\\
&\qquad \times\left[\YYIH{\left[\frac{p^\ast+5}2\right],\left[\frac{p^\ast-1}2\right]}
+\jb{t}^{-1}\EEH{\left[\frac{p^\ast+3}2\right],\left[\frac{p^\ast-1}2\right]}\right].
\end{align}

Recall that $\left[\frac{p^\ast+5}2\right]\le p^\ast$ since $p\ge11$. Overall, for the
interior low energy we have
\begin{align}
I_1 \ \lesssim \ &
\int_0^T
\left[\YYIH{\left[\frac{p^\ast+5}2\right],\left[\frac{p^\ast}2\right]}
+\jb{t}^{-1}\EEH{\left[\frac{p^\ast+3}2\right],\left[\frac{p^\ast}2\right]}\right]\\
&\qquad \times
\left[\YYIH{p^{\ast}+1,p^\ast}+\jb{t}^{-1}\EEH{p^\ast,p^\ast}\right] \EEH{p^\ast,p^\ast} dt\\
 \ \lesssim \ &
\int_0^T
\left[\YYIH{p^\ast,p^\ast-1}
+\jb{t}^{-1}\EEH{p^\ast,p^\ast}\right]\\
&\qquad \times
\left[\YYIH{p^{\ast}+1,p^\ast}+\jb{t}^{-1}\EEH{p^\ast,p^\ast}\right] \EEH{p^\ast,p^\ast} dt\\
\ \lesssim \ &
\int_0^T   \YYIH{ p^\ast, p^\ast - 1 }  \YYIH{  p^\ast + 1 ,  p^\ast   } \EEH{  p^\ast, p^\ast }dt\\
&\qquad +\    \int_0^T  \jb{t}^{-1}   \YYIH{   p^\ast + 1 ,  p^\ast    }  \EE{p^\ast,p^\ast}  dt
\\
&\qquad +\    \int_0^T  \jb{t}^{-2}     {\mathcal E}^{3/2}_{p^\ast,p^\ast}[U](t)  dt.
\end{align}

Next, we are going to estimate these three integrals above. We will use
Theorem \ref{IntDecayNon} and Proposition \ref{HighEnrgEst}. Furthermore,
we will require that $2C_1 \eps < 1$.

The first integral can be estimated as follows:
\allowdisplaybreaks{
\begin{align}
&\int_0^T  \YYIH{ p^\ast, p^\ast - 1 }  \YYIH{  p^\ast + 1 ,  p^\ast   } \EEH{  p^\ast, p^\ast }dt
\\
&\lesssim \ \left(\sup_{0\le t\le T}\jb{t}^{-C_1\eps} \EE{p^\ast, p^\ast}    \right)^{1/2} \\
&\hspace{2cm} \times
 \int_0^T     \jb{t}^{C_1 \eps/2}    \YYIH{ p^\ast, p^\ast - 1 }  \YYIH{  p^\ast + 1 ,  p^\ast   } dt
\\
&\lesssim \  \EEZH{p^\ast, p^\ast}
\left(  \int_0^T
 \jb{t}^{C_1 \eps/2}    \YYI{ p^\ast, p^\ast - 1 } dt \right)^{1/2}
   \left(  \int_0^T     \jb{t}^{C_1 \eps/2}     \YYI{  p^\ast + 1 ,  p^\ast   } dt \right)^{1/2}
\\
&\lesssim \  \EEZH{p^\ast, p^\ast}
\left(\sup_{0\le t\le T} \jb{t}^{-C_1 \eps} \EE{p^\ast, p^\ast}\right)^{1/2}
   \left(\sup_{0\le t\le T}\jb{t}^{-C_1\eps} \EE{p,q}\right)^{1/2}\\
& \lesssim     \EEZ{p^\ast, p^\ast} \EEZH{p, q}.
\end{align}

For the second integral, we have
\begin{align}
\int_0^T  \jb{t}^{-1}  & \YYIH{   p^\ast + 1 ,  p^\ast    }  \EE{p^\ast,p^\ast}  dt\\
&\lesssim
\sup_{0\le t\le T} \jb{t}^{-C_1 \eps} \EE{p^\ast, p^\ast}
\int_0^T  \jb{t}^{-1+C_1{\eps}}   \YYIH{   p^\ast + 1 ,  p^\ast    }    dt
\\
&\lesssim
\EEZ{p^\ast, p^\ast}
\int_0^T  \jb{t}^{-1+C_1{\eps}/2}   \jb{t}^{C_1\eps/2}\YYIH{   p^\ast + 1 ,  p^\ast    }    dt
\\
&\lesssim
\EEZ{p^\ast, p^\ast}
\left(\int_0^T\jb{t}^{-2+C_1{\eps}}dt\right)^{1/2}
\left(
\int_0^T \jb{t}^{C_1{\eps} }
\YYI {   p^\ast + 1 ,  p^\ast    }    dt
\right)^{1/2}
\\
&\lesssim
\EEZ{p^\ast, p^\ast}
\left(\sup_{0\le t\le T}\jb{t}^{-C_1 \eps} \EE{p,q}\right)^{1/2}\\
&\lesssim \EEZ{p^\ast,p^\ast}\EEZH{p,q}.
\end{align}
} 

And finally, the third integral is bounded by
\begin{align}
\int_0^T  \jb{t}^{-2}     {\mathcal E}^{3/2}_{p^\ast,p^\ast}[U](t)  dt
\; &\lesssim\; \left(\sup_{0\le t\le T}\jb{t}^{-C_1{\eps}}\EE{p^\ast,p^\ast}\right)^{3/2}
\int_0^T\jb{t}^{-2+\frac32C_1{\eps}}dt\\
&\lesssim {\mathcal E}^{3/2}_{p^\ast,p^\ast}[U_0] \\
&\lesssim \EEZ{p^\ast,p^\ast}\EEZH{p,q}.
\end{align}

Combining these estimates, we have proven \eqref{GlbI1}, in the case $j=1$.

Now we turn to the estimation of $I_2$.  This is the crucial step, where we use
the specific structure of the nonlinearity.
By the definition of the exterior term
\begin{equation}
I_2 =
  \sum_{\stackthree{\stacktwo{a_1+a_2=a}{k_1+k_2= k}}{|a|+k\le p^\ast}{k\le p^\ast}}
 \left|
 \int_0^T
    \ltip  {  \eta N(  S^{k_1}\Gamma^{a_1} U(t), \nabla S^{k_2}\Gamma^{a_2} U(t)  )   }   {S^k \Gamma^a U(t)}dt
  \right|.
\end{equation}

Referring to the definition of $N(U, \nabla U)$ (see \eqref{DefN},\eqref{DefN1N2}) and applying the constraint
\eqref{ConstrVectrFld1} we note that, in components, the quadratic nonlinear terms are of the form
\begin{equation} \label{ConvectiveTerms}
\begin{aligned}
  N_1( S^{k_1} \Gamma^{a_1}U,\nabla  S^{k_2} \Gamma^{a_2}U)_{ij}
         = &(S^{k_1} \Gamma^{a_1}G)_{, j} \cdot \nabla (S^{k_2} \Gamma^{a_2}v)_i \\
& - S^{k_1} \Gamma^{a_1}v \cdot \nabla (S^{k_2} \Gamma^{a_2} G)_{ij} \\
N_2( S^{k_1} \Gamma^{a_1}U,\nabla  S^{k_2} \Gamma^{a_2}U)_{i}
         =& (S^{k_1} \Gamma^{a_1}G)_{, j} \cdot \nabla (S^{k_2} \Gamma^{a_2}G)_{ij} \\
&-S^{k_1} \Gamma^{a_1}v \cdot \nabla ( S^{k_2} \Gamma^{a_2} v )_i,
\end{aligned}
\end{equation}
where $G_{, j}$ denotes the  $j^{\text{th}}$ column of the
matrix $G$.   
Here and below, the indices satisfy $a_1 + a_2 = a$, $k_1 + k_2 = k$ with $|a| + k \leq p^\ast$ and $k \leq p^\ast$.
Using the gradient decomposition \eqref{GradDecomp}, we can write
\begin{align}
&| N(  S^{k_1}\Gamma^{a_1} U(t), \nabla S^{k_2}\Gamma^{a_2} U(t)  )|\\
&\hspace{.5in}
\lesssim (|\omega\cdot S^{k_1}\Gamma^{a_1}v|+|S^{k_1} \Gamma^{a_1}G^\top\omega|)\;|\partial_rS^{k_2} \Gamma^{a_2}U|\\
&\hspace{1in}+r^{-1} |S^{k_1}\Gamma^{a_1} U| \; | \Omega S^{k_2}\Gamma^{a_2} U| \\
&\hspace{.5in}
= \left(|\omega\cdot S^{k_1}\Gamma^{a_1}v|+\sum_j|\omega\cdot S^{k_1} \Gamma^{a_1} G_{,j}|\right)
\;|\partial_rS^{k_2} \Gamma^{a_2}U|\\
&\hspace{1in}+r^{-1} |S^{k_1}\Gamma^{a_1} U| \; | \Omega S^{k_2}\Gamma^{a_2} U| \\
&\hspace{.5in}\lesssim |\omega\cdot S^{k_1}\Gamma^{a_1}u| \; |\partial_rS^{k_2} \Gamma^{a_2}U|
+r^{-1} |S^{k_1}\Gamma^{a_1} U| \; | \Omega S^{k_2}\Gamma^{a_2} U|,
\end{align}
where $S^{k_1}\Gamma^{a_1} u$ stands for either of
\begin{equation}
S^{k_1}\Gamma^{a_1} u =
\begin{cases}
(S^{k_1} \Gamma^{a_1} G)_{, j}\\
S^{k_1} \Gamma^{a_1} v
\end{cases}.
\end{equation}

Therefore, we obtain the bound
\begin{align}
 I_2 \lesssim
& \sum_{\stackthree{\stacktwo{a_1+a_2=a}{k_1+k_2= k}}{|a|+k\le p^\ast}{k\le p^\ast}}
\left\{ \int_0^T\int_{\rr^3}
    \eta\;|\omega\cdot S^{k_1}\Gamma^{a_1}u|\;|\partial_rS^{k_2}\Gamma^{a_2}U|   \; |S^k \Gamma^a U|\;dxdt \right.\\
  \\
  \label{I2Bound}
& \hspace{1 cm}
+\left. \int_0^T \int_{\rr^3}
     \eta \;r^{-1} |S^{k_1}\Gamma^{a_1} U(t)| \; | \Omega S^{k_2}\Gamma^{a_2} U(t)| \; |S^k \Gamma^a U(t)|\; dx dt \right\} \\
 \equiv & \;I_2' + I_2''.
\end{align}

By Lemma \ref{ExtCalcIneq}, we can bound $I_2''$ as follows:
\begin{align}
& I_2'' \lesssim
\sum_{\stackthree{\stacktwo{a_1+a_2=a}{k_1+k_2= k}}{|a|+k\le p^\ast}{k\le p^\ast}}
   \int_0^T \jb{t}^{-1} \ltn{ \eta |S^{k_1}\Gamma^{a_1} U(t)| \; | \Omega S^{k_2}\Gamma^{a_2} U(t)| \;}
              \ltn{S^k \Gamma^a U(t)} dt \\
              \label{I2''Bound}
& \hspace{2cm} \lesssim
   \int_0^T \jb{t}^{-2} \EE{p^\ast,p^\ast} \EEH{p,q} dt.
\end{align}

We see that  $I_2'$ is bounded by a sum of terms of the form
\begin{equation}\label{SpecTerm}
\int_0^T  \ltn[L^\infty]{\eta \; \omega \cdot S^{k_1}\Gamma^{a_1} u(t)}
       \;  \ltn {S^{k_2}\Gamma^{a_2+1} U(t) \;}
              \ltn{S^k \Gamma^a U(t)} dt,
\end{equation}
Recalling the constraints \eqref{ConstrVectrFld1},\eqref{ConstrVectrFld2},
we see that Proposition  \ref{SpecialSob} can be used to estimate
 $\ltn[L^\infty]{\eta \; \omega \cdot S^{k_1}\Gamma^{a_1} u(t)}$.

In the case $k_1 + |a_1| \leq \left[ \frac{p^\ast}{2} \right]$ and
$k_2 + |a_2| \leq p^\ast$, Proposition  \ref{SpecialSob} gives
\begin{align}
 &  \ltn[L^\infty]{\eta \; \omega \cdot S^{k_1} \Gamma^{a_1} u(t)} \ltn{S^{k_2} \Gamma^{a_2+1} U(t)} \\
& \hspace{3cm} \lesssim
   \jb{t}^{-3/2} \EEH{\left[ \frac{p^\ast+4}{2} \right], \left[ \frac{p^\ast}{2} \right] }
              \EEH{p^\ast+1,p^\ast } \\
& \hspace{3cm} \lesssim
   \jb{t}^{-3/2} \EEH{p^\ast,p^\ast } \EEH{p,q}.
\end{align}

And in the case $k_2 + |a_2| \leq \left[ \frac{p^\ast-1}{2} \right]$ and
$k_1 + |a_1| \leq p^\ast$, again using Proposition  \ref{SpecialSob}, we have
\begin{align}
& 
   \ltn[L^\infty]{\eta \; \omega \cdot S^{k_1} \Gamma^{a_1} u(t)} \ltn{S^{k_2} \Gamma^{a_2+1} U(t)} \\
& \hspace{3cm} \lesssim
   \jb{t}^{-3/2} \EEH{p^\ast+2,p^\ast }
            \EEH{\left[ \frac{p^\ast+1}{2} \right], \left[ \frac{p^\ast-1}{2} \right] } \\
& \hspace{3cm} \lesssim
   \jb{t}^{-3/2} \EEH{p,q} \EEH{p^\ast,p^\ast }.
\end{align}

From the two cases above and \eqref{SpecTerm} we conclude that
\begin{equation}\label{I2'Bound}
I_2' \lesssim \int_0^T \jb{t}^{-3/2} \EE{p^\ast,p^\ast } \EEH{p,q} dt.
\end{equation}

Therefore, from \eqref{I2Bound},\eqref{I2''Bound}, and \eqref{I2'Bound}, we have shown that
\begin{equation}
I_2 \lesssim \int_0^T \jb{t}^{-3/2} \EE{p^\ast,p^\ast } \EEH{p,q} dt.
\end{equation}
Now by Proposition \ref{HighEnrgEst}, we obtain
\begin{equation}
\label{GlbI2}
I_2\lesssim\EEZ{p^\ast,p^\ast } \EEZH{p,q}\int_0^T\jb{t}^{-3/2(1-C_1\eps)}dt\lesssim\EEZ{p^\ast,p^\ast } \EEZH{p,q},
\end{equation}
provided $C_1\eps<1/3$.  This verifies \eqref{GlbI1}, in the case $j=2$.

 By 
\eqref{GlbEnrgInq}, \eqref{GlbI1I2Inq}, and the
estimates \eqref{GlbI1},  the proof of Proposition \ref{HighEnrgEst} is complete.

\end{proof}

\bibliography{JKS}
\bibliographystyle{amsplain}

\end{document}